\documentclass[a4paper,11pt,reqno]{amsart}
\usepackage{a4wide}

\usepackage{fourier}
\usepackage{amsmath}
\usepackage{amsfonts}
\usepackage{amsthm}
\usepackage{cite}
\usepackage{hyperref}

\newcommand{\R}{\mathbb{R}}

\newcommand{\Z}{\mathbb{Z}}

\newtheorem{theorem}{Theorem}[section]
\newtheorem{lemma}[theorem]{Lemma}
\newtheorem{corollary}[theorem]{Corollary}
\newtheorem{prop}[theorem]{Proposition}

\newtheorem{definition}{Definition}

\newtheorem{remark}{Remark}
\newtheorem{example}{Example}

\newcommand{\1}{\mathbf{1}}
\newcommand{\veps}{\varepsilon}
\newcommand{\eps}{\epsilon}

\author[A. Ciomaga]{Adina Ciomaga$^*$}
\address{$^*$ Adina Ciomaga: 
Université de Paris, Laboratoire Jacques-Louis Lions (LJLL), F-75013 Paris, France\\
Sorbonne-Université, CNRS,  LJLL, F-75005 Paris, France\\
Octav Mayer Institute of Mathematics, Romanian Academy, Ia\c si Branch, Romania}
\email{adina@ljll.math-univ-paris-diderot.fr}

\author[D. Ghilli]{Daria Ghilli$^\S$}
\address{$^\S$ Daria Ghilli: Dipartimento di matematica Tullio Levi Civita, 
Università degli studi di Padova, Via Trieste 63 35131, Padova, ITALIA}
\email{ghilli@math.unipd.it}

\author[E. Topp]{Erwin Topp$^{\#}$}
\address{$^{\#}$ Erwin Topp:
Departamento de Matem\'atica y C.C., 
Universidad de Santiago de Chile,
Casilla 307, Santiago, CHILE.}
\email{erwin.topp@usach.cl}

\title[]{Periodic Homogenization for Weakly Elliptic Hamilton-Jacobi-Bellman Equations with Critical Fractional Diffusion}

\begin{document}

\date{\today}

\begin{abstract} 
In this paper we establish periodic homogenization  for Hamilton-Jacobi-Bellman (HJB) equations, associated to nonlocal operators of integro-differential type. We consider the case when the fractional diffusion  has the same order as the drift term, and is weakly elliptic. The outcome of the paper is two-fold. One one hand, we provide Lipschitz regularity results for weakly elliptic nonlocal HJB, extending the results previously obtained in \cite{Bcci12}. On the other hand, we establish a convergence result, based on half relaxed limits and a comparison principle for the effective problem. The latter strongly relies on the regularity and the ellipticity properties of the effective Hamiltonian, for which a fine Lipschitz estimate of the corrector plays a crucial role. 
\end{abstract}

\maketitle
\date{\today}

\bigskip

\bigskip 

\bigskip\section{Introduction}\label{sec:Intro}

In this paper we are interested in periodic homogenization of parabolic nonlocal Hamilton-Jacobi equations of the form
\begin{equation}\label{eq:HJ-eps}
	\left\{
	\begin{array}{ll}
	\displaystyle u_t^\veps(x,t) + H(x, \frac{x}{\veps}, Du^\veps,u^\veps(\cdot,t)) = 0 
					& \text{ in } \R^d\times (0,T)\\
	u(x,0) = u_0(x) 		& \text{ in } \R^d,
	\end{array}
	\right.
\end{equation}
where $T>0$, the initial condition $u_0:\R^d\to\R$ is a bounded uniformly continuous function and $H$ is a continuous Hamiltonian, periodic with respect to its fast variable $\xi=x/\eps$. The unknown functions $u^\veps:\R^d\times[0,T]\to\R$ depend on a homogenization scale $\veps>0$. 
The function $H = H(x,\xi, p,\phi)$ is a Hamilton-Jacobi-Bellman operator $H:\R^d\times\R^d\times\R^d\times\mathcal C^2(\R^d)\cap L^\infty(\R^d)\to\R$, depending non-locally on a function $\phi\in C^2(\R^d)\cap L^\infty(\R^d)$, through an integro-differential operator associated to L\'evy processes. More precisely, given a compact metric space $\mathcal A$, the Hamiltonian takes the form
\begin{equation}\label{eq:defH}
	H(x, \xi, p, \phi) = \sup_{a \in \mathcal A} \Big\{ -\mathcal{L}^a(x, \xi,\phi) - b^a(x, \xi) \cdot p - f^a(x, \xi) \Big\}.
\end{equation}
The integro-differential operator is given by
\begin{equation}\label{eq:NL-L}
	\mathcal{L}^a(x,\xi,\phi) = \int_{\R^d} \big( \phi(x + z) - \phi(x) - \mathbf{1}_B(z) D\phi(x)\cdot z \big) K^a(\xi, z)dz,
\end{equation}
where $ \mathbf{1}_B$ denotes the indicator function of the unit ball $B$ in $\R^d$,
and $K^a(\cdot) = K(a,\cdot)$ is a family of kernels generated by a continuous function $K:\mathcal A\times \R^d\times\R^d\to\R_+$. The kernels are possibly singular at the origin, satisfying the uniform L\'evy condition
\begin{equation*}
	\sup_{a\in \mathcal A} \sup_{{\xi \in \R^d}} \int_{\R^d} \min\left(1, |z|^2\right) K^a(\xi, z)dz < + \infty.
\end{equation*}
Similarly  to $(K^a)_{a\in\mathcal A}$, the families of functions  $(f^a)_{a\in\mathcal A}$ and $(b^a)_{a\in\mathcal A}$ are given respectively by  $f:\mathcal A\times\R^d\times\R^d\to\R$ and $b:\mathcal A\times\R^d\times\R^d\to\R^d$, bounded and continuous functions. \smallskip

Nonlocal equations find applications in mathematical finance and occur in the theory of L\'evy jump-diffusion processes. The theory of viscosity solutions has been extended for a rather long time to integro-differential equations. Some of the first papers are due to Soner  \cite{S86a,S86b} in the context of stochastic control jump diffusion processes. The connection of such nonlocal equations with deterministic and stochastic singular perturbations of optimal control problems appears in \cite{Ab01}, ~\cite{Bs18}, \cite{Bcs16}. Existence and comparison results for second order  degenerate Hamilton-Jacobi-Bellman equations were provided by Benth, Karlsen and Reikvam in \cite{BKR01}. The viscosity theory for general partial integro-differential operators has been recently revisited and extended to solutions with arbitrary growth at infinity by Barles and Imbert \cite{Bi08}. \smallskip

In this paper, we deal with Hamilton-Jacobi-Bellman equations where the diffusion is given by a general L\'evy nonlocal operator, with a kernel depending on the space variable $x$ and we would like to place ourselves in a ``critical'' regime, where both the nonlocal diffusion and the Hamiltonian are of order $1$.  A key issue is the establishment of the concept of the ``order" of the diffusion. It is known  \cite{Cs09} that  the behaviour of the kernel near the origin determines such an order. The typical example of an integro-differential operator of order $1$ is given by  {\em the square root of the Laplacian}, whose kernel  $K(\xi,z) = 1/|z|^{d+1}$ is symmetric, and  independent of $\xi$ : 
\begin{eqnarray*}
(-\Delta)^{1/2} u(x)  & = & \int_{\R^d} (u(x + z) - u(x) - \mathbf{1}_B Du(x) \cdot z)|z|^{-(d+ 1)}dz \\
			 & = & \mathrm{ P.V. } \int_{\R^d} (u(x + z) - u(x))|z|^{-(d + 1)}dz,
\end{eqnarray*}
where $\mathrm{P.V.}$ stands for the Cauchy Principal Value, see~\cite{DNpv12}. 
More generally, {\em uniformly elliptic} kernels could be considered,  i.e.  kernels for which there exist a constant $C_K>0$ such that 
\begin{equation}\label{eq:K-elliptic}
\frac{1}{C_K |z|^{d+1}} \leq K^a(\xi,z) \leq \frac{C_K}{|z|^{d+1}} \quad \mbox{ for all} \ z \in B \setminus \{ 0 \}.
\end{equation}
In the ``critical" regime of uniformly elliptic kernels satisfying equation ~\eqref{eq:K-elliptic}, the nonlocal and gradient terms in~\eqref{eq:defH} have the same scaling properties, and therefore the diffusive role of $\mathcal{L}^a$ enters into competition with the transport effect of the drift term. The critical regime was already studied by Silvestre in ~\cite{S11} and \cite{S12}, where regularity of solutions is shown and the result is used to establish the existence of classical solutions. The above ellipticity assumption is the equivalent of its local version,  which roughly speaking  requires  all the eigenvalues associated to the diffusion matrix to stay bounded away from zero. 
We aim at dealing with more general kernels, where the pointwise ellipticity assumption \eqref{eq:K-elliptic} is replaced by an integral condition. We require kernels to be  {\em weakly elliptic} only, i.e. there exists a  constant $C_K>0$  such that for any given direction $p\in\R^d$, there exist  an {\em ellipticity cone}  $\mathcal C_{\eta,\rho}(p):= \{z\in B_\rho; (1-\eta)|z||p|\leq|p\cdot z|\}$ of aperture $\eta\in(0,1)$ where
	\[ \int_{\mathcal C_{\eta, \rho}(p)}|z|^2 K^a(\xi,z)dz\geq C_K \eta^{\frac{d-1}{2}}\ \rho, \text{ for any } \xi\in\R^d.\]
Here, the quantity $\eta^{\frac{d-1}{2}}$ measures the volume of the cone in the unit ball relative to the volume of the unit ball, while $\rho$ is related to the order/scaling of the nonlocal operator (see Example 1 in \cite{Bcci12} for more details). In particular, any uniformly elliptic operator is weakly elliptic. Solutions associated with this type of weakly elliptic kernels are shown to be Lipschitz \cite{Bcci12} in the case when the nonlocal diffusion has order  larger than $1$; nonetheless, the critical case remained open.\smallskip

The setup we consider is in striking contrast with previous available results in homogenization of integro-differential problems. In~\cite{A09, A12}, Arisawa analyzed periodic homogenization for equations with purely  L\'evy operators, and rather light interaction between the slow and fast variable. Homogenization results for nonlocal equations with variational structure have been recently studied  in ~\cite{Frs17, Pz17}. This paper is closely related to \cite{S10}, where periodic homogenization for uniformly elliptic Bellman-Isaacs equations was obtained by Schwab. Later on these results were extended to stochastic homogenization in \cite{S09}. The arguments in both papers are completely different than ours,  and are based on the obstacle problem method, previously introduced in \cite{CSW05,CS10} in order to establish stochastic homogenization and rates of convergence for fully nonlinear, uniformly elliptic partial differential equations.  Periodic homogenization for nonlocal Hamilton-Jacobi equations with coercive gradient terms has been addressed in \cite{Bct19}, where techniques similar to ours appear, except that here we cannot rely on the gradient coercivity. \smallskip
 
We show that the family of solutions $\big( u^\veps \big)_\veps$ of the Cauchy problem~\eqref{eq:HJ-eps} converges locally uniformly on $\R^d\times[0,T]$, as $\veps \to 0$, to the solution $u$ of an {\em effective problem}
\begin{equation}\label{eq:HJ-eff}
	\left\{
	\begin{array}{ll}
	\displaystyle u_t (x,t) + \overline H(x, Du,u(\cdot,t)) = 0 & \text{ in } \R^d\times (0,T)\\
	u(x,0) = u_0(x) 							& \text{ in } \R^d,
	\end{array}
	\right.
\end{equation}
where the limiting Hamiltonian $\overline H$ is to be implicitly defined. This main result is presented in Theorem~\ref{thm:conv}.
The program is classical, and falls into the lines of the celebrated preprint of Lions Papanicolau and Varadhan \cite{Lpv86} and the seminal papers of Evans \cite{E89, E92}. We  write the oscillatory solution as 
$ u^\veps(x,t) = \bar u(x,t) + \veps \psi(x/\veps) + \cdots,$
and find the effective Hamiltonian $\overline H$ by solving a {\em cell problem} whose solution is the (periodic) corrector $\psi$, then establish properties of $\overline H$ that ensure well-posedness  of the limiting problem \eqref{eq:HJ-eff} and finally  conclude the convergence. \smallskip

Though the result itself is standard in periodic homogenization, a series of difficulties arise, due to the general form and weak ellipticity of the nonlocal operator \eqref{eq:NL-L}: (i) the implicit definition of $\overline H$ which does not say much about its  nonlocal dependence on the whole function $u$, (ii) the absence of comparison principles for equations with integro-differential operators having  general $x$- dependent kernels, and in particular the lack of comparison results for the limiting problem and (iii) the lack of Lipschitz regularity of the oscillatory solutions and of the corrector. We discuss each of these points in turn and the interplay in-between. \smallskip

Homogenization occurs in two steps. The first step is the study of  the cell problem and accordingly the construction the {\em effective  Hamiltonian} $\overline H$, which here reads: given $x, p \in \R^d$ and a function $u\in \mathcal C^2(\R^d)\cap L^\infty(\R^d)$ show that there exists a unique constant $\lambda = \overline H(x,p,u)$ so that the following problem has a Lipschitz continuous, periodic, viscosity solution
\begin{equation}\nonumber 
	\sup_{a\in \mathcal A} \{ - {\mathcal I^a(\xi,\psi)} - \tilde b^a(\xi;x) \cdot D\psi(\xi) - \tilde f^a(\xi;x,p,u) \} = \lambda \quad \text{ in } \R^d,
\end{equation}
where $\tilde b^a$ and $\tilde f ^a$ are to be computed. We note that, in the critical case, general non-symmetric kernels give rise to an extra drift term in the cell problem and $\tilde b^a = \tilde b^a + b_K$, for some $b_K$ carefully determined from the properties of $K$, and this is due to the presence of the compensator term  $\mathbf{1}_B(z) Du(x) \cdot z$ in  \eqref{eq:NL-L}. Contrarily, in the case of symmetric kernels, the lack of the compensator term keeps the drift term unchanged. 
In both scenarios, we give a Lipschitz regularity result for the corrector, with a fine estimate of the Lipschitz seminorm. This will play a crucial role in establishing properties of the effective Hamiltonian, which themselves have an important echo in the proof of convergence. \smallskip

Several properties of the original Hamiltonian $H$ given by~\eqref{eq:defH} are translated into the effective one. If on one hand it is natural that $\overline H$ inherits the nonlocal nature in its third variable, on the other hand no explicit formula can be obtained in general. Some examples of explicit nonlocal effective equations can be found in~\cite{Bct19} and~\cite{Kpz19}, but we stress that these methods cannot be applied in the setting and/or the generality presented here. 
In particular, we establish a non-trivial ellipticity-growth condition for $\overline H$  that further allows to manipulate the effective problem in spite of not knowing its explicit form. \smallskip

The second step is solving the effective problem \eqref{eq:HJ-eff} and showing the convergence of the sequence $\big( u^\veps \big)_\veps$. Well posedness for the limit problem~\eqref{eq:HJ-eff} is not obvious, in view of the absence of explicit formulas for $\overline H$ and the lack of general comparison results for nonlocal problems with $x$-dependent kernels. This is overcome  by a {linearization} of the effective Hamiltonian $\overline H$ via the extremal Pucci operators, and  is  intimately related to the Lipschitz regularity of the corrector and the ellipticity growth property of the effective Hamiltonian. Once comparison for the effective problem is proven, the homogenization result is standard and it follows from the perturbed test function method applied to half relaxed limits.  \smallskip

As pointed out above, both solving the cell problem and  showing the convergence  requires Lipschitz regularity of solutions.To the best of our knowledge, no Lipschitz regularity result had been proven before for this kind of equations in their full generality. In \cite{Bcci12}, Lipschitz regularity is proven for equations involving fractional diffusions with order  in the whole range $(1,2]$, except when the order is one. We complete these results and establish Lipschitz regularity of solutions by Ishii-Lions method, making  use of a non standard test function which behaves radially like $r+r\log^{-1}(r)$. We give a rather general Lipschitz regularity result for weakly elliptic integro-differential operators, which has an interest in its own, extending to the critical case Lipschitz estimates obtained in \cite{Bcci12}. \smallskip

We stress that the methods presented in this article can be extended to other nonlocal homogenization problems and they are not exclusively circumscribed to the critical case described here. We emphasize on the ``linearization" of the effective Hamiltonian, which reveals important information about the limiting  problem. Related to this, it would be interesting to describe the effective problem in terms of an associated optimal control problem. This has been addressed in the deterministic case via the so-called limit occupational measures, see~\cite{Bt15, T11} and references therein. Finally, note that the results presented do not rely on the convexity of $H$, and therefore they can be readily adapted to Hamiltonians $H$ of Bellman-Isaacs type, related to differential games (see~\cite{Bcd97}). \smallskip

The paper is organized as follows: in Section~\ref{sec:prelim} we introduce some notation and define the notion of solution to our problems. In Section~\ref{sec:regularity} we establish a Lipschitz regularity result for  integro-differential equations dealing with nonlocal L\'evy operators of order one. In Section \ref{sec:eff-H} we solve the cell problem and provide useful regularity and ellipticity properties of the effective Hamiltonian. In Section~\ref{sec:homog} we  establish the homogenization result associated to equation~\eqref{eq:HJ-eps}.

\section{Preliminaries and assumptions.}\label{sec:prelim}

\subsection{Notations}
We denote the $d-$dimensional Euclidean space by $\R^d$, and by $\Pi^d = \R^d/\Z^d$ the thorus on $\R^d$. For $x \in \R^d$ and $\rho > 0$ we denote $B_\rho(x)$ the ball centered at $x$ with radius $\rho$, and we simply write $B$ if $x = 0$ and $\rho = 1$. We use the notation $\1_B$ for the indicator function of the unit ball $B$ in $\R^d$. By abuse of notation, we denote the cylinder  $B_\rho(x,t) := B_\rho(x)\times (t-\rho,t+\rho)$. For a metric space $X$ we denote respectively $USC(X )$ and $LSC(X)$ the sets of real-valued upper and lower semicontinuous functions on $X$, 
 $BUC(X )$ the set of bounded uniformly continuous real-valued functions on $X$. The set of $\tau-$ H\"older functions on $X$ is written $\mathcal C^{0,\tau}(X)$, the set of continuous functions is written $\mathcal C(X)$ and we denote $\mathcal C^r(X)$ the set of functions, with continuous differentials of order $r>0$.  The space of essentially bounded measurable functions on $X$ is denoted $L^\infty(X)$ and its norm $||\cdot||_\infty$.

\subsection{Viscosity solutions}

To cope with the difficulties imposed by behaviour of the measure at infinity,  as well as its singularity at the origin, we often split the nonlocal term into 
\[ \mathcal L(x,\xi,\phi) = \mathcal L[{B_\rho}](x,\xi,\phi) + \mathcal L[{B_\rho}^c](x,\xi,\phi),\]
with $0<\rho<1$, where  for any $D\subset\R^d$ measurable, we write 
\begin{equation*}
\mathcal L[D](x,\xi,\phi) = \int_D \big( \phi(x + z) - \phi(x) - \mathbf{1}_B(z) D\phi(x)\cdot z \big) K^a(\xi, z)dz.
\end{equation*}
We work in the setting of viscosity solutions, as described in \cite{Bi08}.  In this setup, the nonlocal term is evaluated in terms of a smooth test function on $B_\rho$ and on the function itself on $B_\rho^c$. We give below the definition for a slightly modified equation
\begin{equation}\label{eq:HJ}
	\left\{
	\begin{array}{ll}
	\displaystyle u_t(x,t) + \mathcal H(x, Du,u) = 0 
					& \text{ in } \R^d\times (0,T)\\
	u(x,0) = u_0(x) 		& \text{ in } \R^d,
	\end{array}
	\right.
\end{equation}
where $\mathcal H$ is to be properly defined in each context (for the original oscillating problem \eqref{eq:HJ-eps}, for the cell problem \eqref{eq:cell}, and for the limiting problem \eqref{eq:HJ-eff}).
\begin{definition}[\em Viscosity solutions]\label{def:H-sol} $\;$
\begin{enumerate}
\item We say an upper semi-continuous (usc) function $u:\R^d\times (0,T]\rightarrow\R$ is 
	a {\em viscosity subsolution} of \eqref{eq:HJ} iff for any $\phi\in\mathcal C^{2}(\R^d\times[0,T])$, 
	if $(x,t)$ is a maximum of $u-\phi$ in  $B_\rho(x,t)$ then
	\begin{eqnarray*}
	\phi_t (x,t) + \mathcal H(x, D\phi(x,t), 
		\mathbf{1}_{B_\rho(x)} \phi (\cdot,t) +  \mathbf{1}_{B^c_\rho(x)} u (\cdot,t)) \leq 0.
	\end{eqnarray*}
\item We say a lower semi-continuous (lsc) function $u:\R^d\times (0,T]\rightarrow\R$ is 
	a {\em viscosity supersolution} of \eqref{eq:HJ} iff for any $\phi\in\mathcal C^{2}(\R^d\times[0,T])$, 
	if $(x,t)$ is a minimum of $u-\phi$ in  $B_\rho(x,t)$ then 
	\begin{eqnarray*}
	\phi_t (x,t) +  \mathcal H(x, D\phi(x,t), 
		\mathbf{1}_{B_\rho(x)} \phi (\cdot,t) +  \mathbf{1}_{B^c_\rho(x)} u (\cdot,t)) \geq 0.
	\end{eqnarray*}
\item We say $u$ is a {\em viscosity solution} if it is both a viscosity  subsolution and supersolution.	
\end{enumerate}
\end{definition}	

This definition has been formulated so it literally applies to the effective Hamiltonian $\overline H$, provided we show before hand that $\overline H$ is well defined. A similarly definition can be given for the stationary case and henceforth, for the cell-problem.

\subsection{Formal expansion}
In order to introduce the set of assumptions,  and make precise our results we begin with the usual formal asymptotic expansion	
\[ u^\veps(x,t)=\bar u(x,t)+\veps \psi\left(\frac{x}{\veps}\right) + ... \]
where $\bar u(x,t)$ is the average profile and $\psi(\xi)$ is the  periodic corrector. Though this computation already appears in \cite{Bct19}, for the readers' convenience we develop it here, in order to emphasize on (i) the interference between the order of the nonlocal operator and the homogenization scale $\veps$ and (ii) the need to distinguish within the set of assumptions between the symmetric and non-symmetric case and  the fact that in the case of non-symmetric kernels the expansion gives rise to an extra drift term in the corrector equation. 

Plugging the previous expression into the nonlocal term, it follows that
\begin{eqnarray*}
	\mathcal{L}^a\left(x, \frac{x}{\veps}, u^\veps(\cdot,t)\right)
	&=&\int_{\R^d} \left(u^\veps(x+z,t)-u^\veps(x,t)-\1_B(z)Du^\veps(x,t) \cdot z\right)
		K^a\left(\frac{x}{\veps},z\right)\; dz\\\nonumber 
	&=& \int_{\R^d} \left(\bar u(x+z,t)-\bar u(x,t)-\1_B(z)D\bar u(x,t) \cdot z\right)
		K^a\left(\frac{x}{\veps},z\right)\; dz +\\ 
	&&\veps \int_{\R^d} \left(\psi\left(\frac{x+z}{\veps}\right)-\psi\left(\frac{x}{\veps}\right)-\1_B(z)D\psi\left(\frac{x}{\veps}\right) \cdot \frac{z}{\veps}\right)K^a\left(\frac{x}{\veps},z\right)\; dz.
\end{eqnarray*}
Therefore, denoting the fast variable $ x/\veps = \xi$, we can write the nonlocal term as
\[ \mathcal{L}^a\left(x, \xi, u^\veps(\cdot,t)\right) =  \mathcal{L}^a\left(x,\xi, \bar u(\cdot,t) \right)+ \mathcal I_\veps^a(\xi,\psi), \]
where
\[ \mathcal I_\veps^a(\xi,\psi)  = \veps^{d+1} \int_{\R^d} 
	\left( \psi(\xi + z) - \psi(\xi) - \1_{B_{1/\veps}} D\psi(\xi) \right) K^a(\xi,\veps z)\; dz.\]
To keep the ideas clear in this formal expansion assume the kernel is of the following form, regardless its symmetry
\[K^a(\xi,z) = \frac{k(\xi,z)}{|z|^{d+1}}.\]
Note further that
\begin{enumerate}
\item[(i)] if $k^a(\xi,\veps z) = k^a(\xi)$,
	the compensator term in the nonlocal expression $\mathcal J_\veps^a(\xi,\psi)$  vanishes and 
	\[\mathcal I_\veps^a(\xi,\psi) = \mathcal L^a(\xi,\xi,\psi) = k^a(\xi) (-\Delta)^{1/2}\psi(\xi).\]
\item[(ii)] if $k^a(\xi,\veps z)$ is not independent of $z$, 
	we employ a modulus of continuity of $k$ 
		\[\omega_k(r) = \sup_{a\in\mathcal A}\sup_{\xi\in \Pi^d} \sup_{|z|\leq r} \left|k^a(\xi,z) - k^a(\xi,0)\right|.\]
	to separate the nonlocal term into
	\[ \mathcal I_\veps^a(\xi, \psi)= k^a(\xi, 0)(-\Delta)^{1/2} \psi(\xi)+ \mathcal J_\veps^a(\xi, \psi),\]
	where
	\[ \mathcal J_\veps^a(\xi, \psi)=\int_{\R^d} \left (\psi(\xi+z)-\psi(\xi)-\1_{B_{1/\veps}}(z) D\psi(\xi) \cdot z\right)
	 \frac{k^a(\xi, \veps z)-k^a(\xi,0)}{|z|^{d+1}}\; dz. \]
	The term $\mathcal J_\veps^a(\xi, \psi)$ can be split into
	\[ \mathcal J_\veps^a(\xi, \psi)=\mathcal J_\veps^a[B](\xi, \psi)+
	   \mathcal J_\veps^a[B_{1/\veps}\setminus B](\xi, \psi)+\mathcal J_\veps^a[B_{1/\veps}^c](\xi, \psi),\] 
	   where we use the notation $\mathcal J[D]$ to indicate the domain on which the integral is computed.
	Assuming that $\psi\in\mathcal C^2(\R^d)\cap L^\infty(\R^d)$ with bounded $||D\psi||_\infty$ and $||D^2\psi||_\infty$ the following estimates hold	
	\begin{eqnarray*}
	|\mathcal J_\veps^a[B](\xi, \psi)| 
		& \leq &\frac12 ||D^2\psi||_\infty \int_B |z|^2\frac{|k^a(\xi, \veps z)-k^a(\xi, 0)|}{|z|^{d+1}}\; dz \\ 
		& \leq &\frac12 ||D^2\psi||_\infty \omega_k(\veps)\int_B|z|^2\frac{dz}{|z|^{d+1}}=o_\veps(1),\\
	|\mathcal J_\veps^a[B_{1/\veps}^c](\xi, \psi)|
		& \leq & 4||\psi||_\infty ||k||_\infty \int_{B_{1/\veps}^c} \frac{dz}{|z|^{d+1}}=o_\veps(1),
	\end{eqnarray*}
	whereas	
	\begin{eqnarray*}
	\mathcal J_\veps^a[B_{1/\veps}\setminus B](\xi, \psi)
		&=&\int_{B_{1/\veps}\setminus B}\left(\psi(\xi+z)-\psi(\xi)\right)\frac{k^a(\xi, \veps z)-k^a(\xi, 0)}{|z|^{d+1}}\; dz+\\ 
		&&\int_{B_{1/\veps}\setminus B} D\psi(\xi)\cdot z \frac{k(\xi, \veps z)-k^a(\xi, 0)}{|z|^{d+1}}\; dz\\ 	
		&=&o_\veps(1)+\int_{B\setminus B_\veps} D\psi(\xi) \cdot z \frac{k^a(\xi, z)-k^a(\xi,0)}{|z|^{d+1}}\; dz\\ 
		&=&o_\veps(1)+D\psi(\xi)\cdot b_K^a(\xi), 
	\end{eqnarray*}
	where 
	\[b^a_K(\xi)=\int_{B}(k^a(\xi, z)-k^a(\xi, 0))\frac{z}{|z|^{d+1}}\; dz\]
	is well-defined provided that
	$\displaystyle \int_0^1 \frac{\omega_k(r)}{r}dr  < \infty.$
	To conclude, we have that
	\[ \mathcal I_\veps^a(\xi,\psi) = k^a(\xi,0)(-\Delta)^{\frac{1}{2}}\psi(\xi)+D\psi(\xi)\cdot b_K^a(\xi) +o_\veps(1). \]
\end{enumerate}

Plugging everything in \eqref{eq:HJ-eps}, we arrive to the following equation which must be satisfied both with respect to the slow variable $x$ and the fast variable $\xi$ simultaneously
\begin{eqnarray*}
u_t(x,t) + \sup_{a\in\mathcal A} \Big\{ & - & \mathcal  L^a(x,\xi,\bar u(\cdot,t)) - k^a(\xi,0) (-\Delta)^{1/2} \psi(\xi)\\
& - & b^a(x,\xi)\cdot D\bar u(x,t) - \left(b^a(x,\xi) + b_K(\xi)\right)\cdot D\psi(\xi) - f^a(x,\xi) \Big\} = 0.
\end{eqnarray*}
We are lead, in this context, to solving first the following cell problem:
given $x, p \in \R^d$ and a function $u\in \mathcal C^2(\R^d)\cap L^\infty(\R^d)$ show that there exists a unique constant $\lambda\in\R$ so that the following problem has a Lipschitz continuous, periodic, viscosity solution
\[\sup_{a\in \mathcal A} \{ - k^a(\xi,0)(-\Delta)^{1/2}\psi(\xi) - \tilde b^a(\xi;x) \cdot D\psi(\xi) - \tilde f^a(\xi;x,p,u) \} = \lambda,\]
where the source term is given by 
$ \tilde f^a(\xi;x,p,u)  =  f^a(x,\xi) + b^a(x,\xi) \cdot p + \mathcal L^a(x, \xi, u), $
and the drift adds an extra term
$ \tilde b^a(\xi; x)  =  b^a(x,\xi) + b_K^a(\xi).$
The constant $\lambda$ is known in the literature as the effective Hamiltonian and denoted by  $\lambda = \overline H(x,p,u)$. This implicitly defines the effective equation (or the limit equation) \eqref{eq:HJ-eff}, which is shown to be satisfied by the average profile $\bar u$. Once well posedness is established for the effective equation, the convergence of the whole sequence $\big(u^\veps\big)_{\veps>0}$  towards the average profile $\bar u$ is shown.

Going back to the points raised in (i) and (ii), we have seen above that nonlocal terms having kernels with a general dependence on the fast and slow variables give rise to an extra drift term. This is due on one hand to the fact that the homogenization scale $\veps$ has the same order as the nonlocal diffusion (in occurence $1$) and on the other hand to the fact that the kernel has a non-symmetric behaviour in the slow variable. This is not the case if  the kernel is symmetric, when the compensator is not needed. 
\subsection{Assumptions}

Homogenization results are established both  for symmetric  and non-symmetric kernels, though the formal expansion has been given only for the non-symmetric case. To this end, we make two set of assumptions, corresponding to each setup.

\begin{itemize}
\item[$( Ks)$]
	For each $a\in \mathcal A$,  $K^a$ is {\em symmetric} with respect to $z$, 
	i.e. for all $\xi\in\R^d$ and $z\in \ R^d\setminus\{0\}$, 
	\[K^a(\xi,z) = K^a(\xi, -z)\]
	and {\em homogeneous} with respect to $z$, i.e. 
	for all $\xi\in\R^d$, $z\in \ R^d\setminus\{0\}$ and any $\epsilon >0$,
	\[ K^{a}(\xi, \epsilon z) = \frac{1}{\epsilon^{(d + 1)}} K^a(\xi, z).\] 
\item[$( Kns)$]  
	For each $a\in\mathcal A$, there exists 
	$k^a \in \mathcal C(\R^{2d})\cap L^\infty(\R^{2d})$ such that, 
	for all $\xi\in\R^d$ and  $z\in B\setminus\{0\}$,
	\[K^a(\xi,z) = \frac{k^a(\xi,z)}{|z|^{d+1}},\]
	and there exists a constant $C_K >1$ such that 
	\[\sup_{a\in \mathcal A} \sup_{\xi\in\Pi^d}  
		\int_0^1 \sup_{|z|\leq r} \left|k^a(\xi,z) - k^a(\xi,0) \right| \frac{dr}{r} \leq C_K. \]
\end{itemize}

To the scaling and symmetry assumptions above, we add a series of assumptions for the family of L\'evy kernels, in order to ensure periodicity, existence of solutions,  comparison results and regularity. These have now become classical, see \cite{Bi08,Bci11, Bcci12}.
\begin{itemize}
\item [$( K0)$] 
	For any $a\in\mathcal A$, the mapping $\xi\mapsto K^a(\xi,z)$ is $\Z^d$ periodic, for all $z\in\R^d$.
\item [$( K1)$] 
	There exists a constant $ C_K>0$ such that, 
	\[\sup_{a\in\mathcal A}\sup_{\xi\in\R^d} \int_{\R^d} \min(1,|z|^2) K^a(\xi,z)dz\leq C_K. \]
\item [$( K2)$] 
	There exist a constant $C_K>0$ such that for any $p\in\R^d$, there exist a $0<\eta<1$ 
	 such that the following holds for all $a\in\mathcal A$, for any $\xi\in\R^d$ and for all $\rho>0$,
	\[ \int_{\mathcal C_{\eta, \rho}(p)}|z|^2 K^a(\xi,z)dz\geq C_K \eta^{\frac{d-1}{2}}\ \rho,\]
	with $\mathcal C_{\eta,\rho}(p):= \{z\in B_\rho; (1-\eta)|z||p|\leq|p\cdot z|\}$.
\item [$( K3)$] 
	There exist  a constant $C_K>0$ and an exponent $\gamma\in(0,1]$ such that 
	for all $a\in\mathcal A$, for any $\xi_1,\xi_2\in\R^d$ and all $\rho>0$,
	\begin{eqnarray*} 
 	\int_{B_\rho} |z|^2 		
		|K^a(\xi_1,z)-K^a(\xi_2,z)| dz &\leq & C_K |\xi_1-\xi_2|^{\gamma} \rho \\
 	\int_{B\setminus B_\rho}|z| 	
		|K^a(\xi_1,z)-K^a(\xi_2,z)| dz &\leq& C_K |\xi_1-\xi_2|^{\gamma} |\ln\rho|\\
 	\int_{\R^d\setminus B_\rho} 				
		|K^a(\xi_1,z)-K^a(\xi_2,z)| dz&\leq& C_K |\xi_1-\xi_2|^{\gamma} \rho^{-1}.
	\end{eqnarray*}	
\end{itemize}

Finally, we assume the following for the drift term and the running cost.
\begin{itemize}
\item[$( H0)$]
	For each $a\in\mathcal A$, the mappings $\xi\mapsto f^a(x,\xi)$, $\xi\mapsto b^a(x,\xi)$  
	are $\Z^d$ periodic, for all $x\in\R^d$.
\item[$( H1)$]  
	 Let $f^a:\R^{\tilde d}\to\R$ and $ b^a:\R^{\tilde d}\to\R^{\tilde d}$ be two families of bounded functions. 
	 There exist two constants $C_{f}, C_{b} > 0$ and exponents $\alpha,\beta\in(0,1] $ such that, 
	 for all $a\in\mathcal A$ and $x_1,x_2 \in\R^{\tilde d}$, 
	\begin{eqnarray*}
		| f^a (x_1) - f^a (x_2)| \leq C_f  |x_1 - x_2|^{\alpha},\;\;\;
		| b^a(x_1) - b^a(x_2)| \leq  C_b |x_1 - x_2|^{\beta}.
	\end{eqnarray*} 
\end{itemize}
This continuity assumption is a classical condition to conclude the existence of global solutions of Bellman equations related to finite/infinite horizon control problems. We write assumption $( H1)$ in the previous general form, since we  alternatively use it on variables $x$ and $\xi$. 

\subsection{Examples}

Here are some typical examples of kernels that correspond to our setup.
\begin{example} 
Let $( K^a)_{a\in\mathcal A}$ be a family of kernels of the form
\[ K^a(\xi,z) = \frac{1}{ | M^a(\xi)z\cdot z|^{(d+1)/2}}\quad \xi\in\R^d, z\in\R^d\setminus\{0\},\]
where $M^a : \R^d \to \mathbf{S}^d$ is a family of periodic $\mathcal C^1$ matrices,  and with eigenvalues uniformly bounded above and below: there exists $c_K > 1$ such that for each $a\in\mathcal A,\xi\in\R^d$, all the eigenvalues of $M^a(\xi)$ belong to the interval $[1/c_K,c_K]$.
\end{example}

\begin{example} 
Let $( K^a)_{a\in\mathcal A}$ be a family of kernels of the form
\[ K^a(\xi,z) = \frac{k^a(\xi,z/|z|)}{|z|^{d+1}} \quad \xi\in\R^d, z\in\R^d\setminus\{0\},\]
where $k^a : \R^d \times \mathbf{S}^{d-1}\to\R$ is a family of bounded continuous functions, periodic and H\"older continuous with respect to their first variable and symmetric with respect to their second variable.
\end{example}

\begin{example} 
Let $( K^a)_{a\in\mathcal A}$ be a family of kernels of the form
\[ K^a(\xi,z) = \frac{k^a(\xi) e^{-i\pi_i(z)}}{|z|^{d+1}} \quad \xi\in\R^d, z\in\R^d\setminus\{0\},\]
where $k^a : \R^d \to\R$ is a family of bounded  H\"older continuous and periodic functions, and $\pi_i:\R^d\to\R$ is the projection function onto the $i-$th component, $\pi_i(z_1,\cdots,z_d) = z_i$.
\end{example}

Finally, as announced in the introduction, we aim at  dealing with degenerate kernels, such as kernels whose measure is supported only in half space, as in the example below.
\begin{example}
Let $( K^a)_{a\in\mathcal A}$ be a family of kernels of the form
\[ K(\xi,z) = \1_{\{z_i > 0\}} \frac{k^a(\xi)}{|z|^{d + 1}} \quad z \in \R^d\setminus\{0\},\]
where, as before, $k^a : \R^d \to\R$ is a family of bounded  H\"older continuous and periodic functions, and  $z_i$ is the $i-$th component of $z$.
\end{example}

\section{Regularity Estimates.}\label{sec:regularity}

In this section we establish Lipschitz regularity of viscosity solutions of nonlocal Hamilton Jacobi equations, when the order of the integro-differential operator is one. To this end, we apply Ishii-Lions's method, as for previously obtained results in \cite{Bci11, Bcci12}. If in the case of fractional diffusions of order larger than one (also known as subcritical) it was necessary to show first that the solution is $C^{0,\tau}$ for some small $\tau>0$, and employ this estimate to get Lipschitz, the technique failed for the critical case.  We now complete this work and show below that, with a proper choice of control function, Lipschitz estimates can be directly obtained in the critical regime for drift fractional-diffusion equations, and their extension to Bellman equations. This will be further used when solving the cell problem, and establishing the homogenization results. 

Consider for any $\delta\geq 0$, the following stationary problem
\begin{equation}\label{eq:HJ-d}
	\delta u + \mathcal H(x, Du, u) = 0 \quad \mbox{in} \ \R^d,
\end{equation}
where the Hamiltonian takes the Bellman form
\begin{equation}\label{eq:H}
	\mathcal H(x,p,u) = \sup_{a\in\mathcal A} \{ -\mathcal I^a{(x,u)} - b^a(x) \cdot p - f^a(x) \}, 
\end{equation}
with the nonlocal operator given by
\begin{equation}\label{eq:NL-I}
	\mathcal I^a{(x,u)} = \int_{\R^d} \big( u(x + z) - u(x) - \mathbf{1}_B(z)Du(x)\cdot z \big) K^a(x, z)dz.
\end{equation}
The main Lipschitz regularity result is given in the theorem below. Note that we do not assume periodicity. Assumptions $(Ks)$ and $(Kns)$ play no role in establishing the regularity of solutions, whereas the weak regularity assumption $(K2)$ is crucial. 

\begin{theorem}\label{thm:Lip}
Let  $(f^a)_{a\in\mathcal A}$,  $(b^a)_{a\in\mathcal A}$ two families of bounded functions on $\R^d$ satisfying $(H1)$ with H\"older exponents respectively ${\alpha,\beta}\in(0,1]$ and constants $C_f, C_b$, and $(K^a)_{a\in\mathcal A}$ be a family of kernels satisfying $(K1)-(K3)$ with H\"older exponent $\gamma\in(0,1]$ and constant $C_K$. 
Then any viscosity solution $u\in BUC(\R^d)$ of~ \eqref{eq:HJ-d} is Lipschitz continuous, satisfying the following estimate: for every ${{\sigma}} \in (0, {\alpha})$ there exists a constant $C_{{\sigma}} > 0$ such that, for all $x,y\in\R^d$,
\begin{equation}\label{eq:Lip}
	|u(x) - u(y)| \leq C_{{\sigma}} C_f^\frac{1}{1 + {{\sigma}}}  |x - y|.
\end{equation}
The constant $C_\sigma$ depends on ${\alpha}, \|u\|_\infty$, and on the constants $C_f,C_b,C_K$, but is independent of $\delta,\beta,\gamma$.
\end{theorem}

\begin{proof}[Proof of Theorem~\ref{thm:Lip}] 
The method, which has now become classical, consists in shifting the solution $u$ and showing that the corresponding difference can be uniformly controlled by a concave function. This translates into a doubling of variables technique, leading to viscosity solutions equations estimates. The proof will be divided in several steps.\smallskip

\noindent {\em Step 1. Doubling of variables.} Let
\[ \Phi(x,y) = u(x) - u(y) - L\phi(x-y) - \psi_\zeta(x), \]
where $\phi$ is radial function $\phi(z) = \varphi(|z|)$ with a suitable choice of a smooth, increasing, concave function $\varphi$, and $\psi_\zeta$ is a smooth localisation term. The penalization function $\varphi:\R_+\to\R_+$ is given here by 
\begin{align*}\label{eq:phi}
	\varphi(r) = \left\{ \begin{array}{ll} 
		0			& r = 0\\
		r + r\log^{-1}(r) 	& r\in (0,r_0]\\
		\varphi(r_0)	& r\geq r_0,
\end{array} \right.
\end{align*}
where $r_0\in(0, 0.04)$, so that the function $\varphi$ is concave and increasing, and for all $r\in(0,r_0]$,
\begin{eqnarray*} 
	 r/2 < & \varphi(r) & < r, \\
  	 1/2 \leq & \varphi'(r) & < 1 \\
 	 -(r \log^{2}(r))^{-1}  \leq & \varphi''(r) & \leq  -(r \log^{2}(r))^{-1}/2.
\end{eqnarray*} 
The localisation term is given by $\psi_\zeta (x) = \psi(\zeta x)$, where  $\psi \in \mathcal C^2(\R^d;\R_+)$ with bounded $\psi$, $D\psi$ and $D^2\psi$ on $\R^d$, such that 
\[ \psi(x) = \left\{ \begin{array}{ll} 
	0			& |x| \leq 1 \\
	3 \ \mathrm{osc}_{\R^d}(u)	& |x| \geq 2.
\end{array} \right.\]
 
Our aim is to show that there exists an $L>0$ such that
\[  |u(x) - u(y)| \leq  L\phi(x-y)  \text{ if }  |x-y|\leq r_0.\]
We argue by contradiction and assume that, for any choice of $L>2\|u\|_\infty$ large enough, and $\zeta\in(0,1)$ small enough, $\Phi$ has a positive maximum, that we denote
\[ M_L = \sup_{x,y\in\R^d} \Phi(x,y) = \Phi(\bar x,\bar y)>0.\]
To simplify the notation we drop the dependence  on $L$ and $\zeta$ for the point $(\bar x, \bar y)$ where the maximum is attained. It is immediate to see that
\begin{eqnarray}\label{eq:Lip-est-Lp}
	&& L |\bar x - \bar y| /2 \leq L \varphi\left( |\bar x-\bar y| \right) \leq 2 ||u||_\infty,\\ \nonumber
	&& L |\bar x - \bar y| /2 \leq L \varphi\left( |\bar x-\bar y| \right) \leq  \omega_u\left( |\bar x - \bar y| \right),
\end{eqnarray}
where $\omega_u(\cdot) $ is the modulus of continuity of $u$ (the solution being uniformly continuous).
This implies in particular that $|\bar x- \bar y|$ is uniformly bounded above and away from zero as $\zeta\to 0$, and  $|\bar x- \bar y|\to 0$ as $L\to \infty$, but also that $L|\bar x -\bar y| \to 0$ as $L\to\infty$. In addition
\begin{equation}\label{eq:Lip-ineq-M}
	M_L \leq u(\bar x) - u(\bar y)\leq \omega_u(|\bar x-\bar y|).
\end{equation}

\noindent {\em Step 2. The viscosity inequalities}.
Let 
\[ \bar p= \bar x - \bar y,\; 	
   \hat p= \bar p/ |\bar p|,\;  	
   p = D\phi(\bar p) = \varphi'(|\bar p|)\hat p, \; 
   q = D\psi_\zeta(\bar x),\]
\[ \phi_y(x) = L\phi(x- y)  + \psi_\zeta(x)\;\;\;  \text{ and } \;\;\;
   \phi_x(y) = - L \phi(x - y).\]
Note that $u - \phi_{\bar y}$ has a global maximum at $\bar x$, respectively $u-\phi_{\bar x}$ has a global minimum at $\bar y$ and $D\phi_{\bar y}(\bar x) = D\phi_{\bar x} (\bar y) = Lp.$ It follows from the viscosity inequalities that, for any $\nu > 0$, there exists $a \in \mathcal A$ such that, for all $0<\rho'<1$, we have 
\begin{eqnarray*}
&& {\delta u(\bar x)} - \mathcal I^a[B_{\rho'}]{(\bar x,\phi_{\bar y})}-\mathcal I^a[B_{\rho'}^c]{(\bar x,u)} 
	- L b^{a}(\bar x)\cdot p - f^{a} (\bar x) \leq 0\\
&& {\delta u(\bar y)} - \mathcal I^a[B_{\rho'}]{(\bar y,\phi_{\bar x})}- \mathcal I^a[B_{\rho'}^c]{(\bar y,u)}
	- L b^{a}(\bar y)\cdot p - f^{a} (\bar y) > - \nu,
\end{eqnarray*}
where we have used the notation $\mathcal I^a[D]{(x,u)} $ to denote the nonlocal operator \eqref{eq:NL-I} computed on the set $D$. 
Denote
\begin{eqnarray*}
&&	\mathcal T^a[B_{\rho'}]{(\bar x, \bar y,\phi)}  : =
	\mathcal I^a[B_{\rho'}]{(\bar x,\phi_{\bar y})}- 
	\mathcal I^a[B_{\rho'}]{(\bar y,\phi_{\bar x})}   \\
&&	\mathcal T^a[B_{\rho'}^c]{(\bar x, \bar y, u)}  : = 
	\mathcal I^a[B_{\rho'}^c]{(\bar x,u)} -
	\mathcal I^a[B_{\rho'}^c]{(\bar y,u)}.
\end{eqnarray*}
Subtract the two inequalities and use the regularity assumption $(H1)$ and \eqref{eq:Lip-ineq-M}, to get that
\begin{eqnarray}\label{eq:Lip-visc-ineq} \nonumber
\delta M_L -\left(\mathcal T^a[B_{\rho'}]{(\bar x, \bar y,\phi)}  +\mathcal T^a[B_{\rho'}^c]{(\bar x, \bar y, u)}  \right)
	& < &  \nu + L \left(b^a(\bar x) - b^a(\bar y) \right)\cdot p + f^a(\bar x) - f^a(\bar y)\\
	& < & \nu + L C_b |\bar x - \bar y|^{{\beta}}  |p| + C_f |\bar x - \bar y|^{{\alpha}}\\  \nonumber
	& < & \nu + L C_b |\bar p|^{{\beta}}+ C_f |\bar p|^{{\alpha}}.
\end{eqnarray}

\noindent{\em Step 3. The nonlocal estimate}.
We first let $\rho'\to 0$ and see that the term $\mathcal T^a[B_{\rho'}]{(\bar x, \bar y, \phi)} $ is $o_{\rho'}(1)$. We then let $\zeta\to 0$ and we note that the nonlocal terms corresponding to $\psi_\zeta$  are of order $o_\zeta(1)$. In what follows, we drop the dependence and {\em all terms in $\rho'$ and $\zeta$}. To simplify notations, we write $\mathcal T^a{(\bar x, \bar y, u)} $ instead of $\mathcal T^a[\R^d]{(\bar x, \bar y, u)} $. It is useful to already see that the maximum of $\Phi$ gives the following bounds for the expressions in $u$, appearing as the integrant of the nonlocal terms composing $\mathcal T^a{(\bar x, \bar y, u)} $. Namely, for all $z\in\R^d$,
\begin{eqnarray} \label{eq:Lip-diff-u} \nonumber
  	u(\bar x + z) - u(\bar x) - p\cdot z  & \leq & L\left( \phi(\bar p + z) - \phi(\bar p) - p\cdot z \right)\\
 	u(\bar y)-u(\bar y + z)  + p\cdot z & \leq & L\left(\phi(\bar p - z) - \phi(\bar p) + p\cdot z\right).
\end{eqnarray}
Here again, we dropped the terms in $\psi_\zeta$ to simplify the presentation.

It is within the nonlocal difference $\mathcal T^a{(\bar x, \bar y, u)} $ that we will see the role of the critical fractional diffusion in obtaining the right Lipschitz estimates. The key bound comes from the weak ellipticity in the gradient direction, given by assumption $(K2)$. To make  this clear, we proceed as usual (see \cite{Bci11, Bcci12}) and split the nonlocal difference into
\begin{eqnarray}\label{eq:Lip-NLsplit}
 	\mathcal T^a{(\bar x, \bar y, u)}  &  = & 
 	\mathcal T^a[\mathcal C_{\eta,\rho}(\bar p)](\bar x, \bar y)  +
	\mathcal T^a[B_\rho\setminus\mathcal C_{\eta,\rho}(\bar p)](\bar x, \bar y) +\\
	\nonumber &&
	\mathcal T^a[B\setminus B_\rho](\bar x, \bar y) + 
	\mathcal T^a[B^c](\bar x, \bar y),
\end{eqnarray}	
where $ \mathcal C_{\eta,\rho}(\bar p)$  is the ellipticity cone in the direction of the gradient, given by $(K2)$ with $\bar p = \bar x-\bar y$, and $\eta\in(0,1)$ and $\rho>0$  yet to be determined. 

\begin{lemma}[Nonlocal estimate on the ellipticity cone]\label{lemma:estimate-cone}
Assume $(K2)$ holds with the ellipticity cone $\mathcal C_{\eta,\rho}(\bar p)$
and let  $\rho = c_1 |\bar p| \log^{-2}(|\bar p|)$, $\eta = c_2 \log^{-2}(|\bar p|),$
with $c_1,c_2>0$ sufficiently small. Then, there exist a constant $C>0$ such that, for all $a\in\mathcal A$,
\[\mathcal T^a[\mathcal C_{\eta,\rho}(\bar p)](\bar x, \bar y)  \leq - C L \left|\log(|\bar p|) \right|^{-(d+3)}.\]
\end{lemma}

\begin{proof}
Fix $a\in\mathcal A$. Note that, in view of \eqref{eq:Lip-diff-u}, 
\begin{eqnarray*}
\mathcal T^a[\mathcal C_{\eta,\rho}(\bar p)](\bar x, \bar y) 
	& \le & L\int_{\mathcal C_{\eta,\rho}(\bar p)}
		\left( \phi(\bar p + z) - \phi(\bar p) -D\phi(\bar p) \cdot z\right)K^a(\bar x, z)dz + \\
	&     & L\int_{\mathcal C_{\eta,\rho}(\bar p)}
	 	\left( \phi(\bar p - z) - \phi(\bar p) + D\phi(\bar p) \cdot z\right)K^a(\bar y, z)dz.
\end{eqnarray*} 
Using Taylor's integral formula, the term above can be further bounded by
\[\mathcal T^a[\mathcal C_{\eta,\rho}(\bar p)](\bar x, \bar y) \leq \sup_{a\in\mathcal A}  \frac{L}{2} \int_{\mathcal C_{\eta,\rho}(\bar p)}\sup_{|s|\leq 1}\left(D^2 \phi(\bar p+ s z) z\cdot z \right) (K^a(\bar x,z) + K^a(\bar y,z)) dz.\]

\noindent Recall that $\phi(z) = \varphi(|z|)$ and use the notation $\hat z = z/|z|$. It follows that
\begin{eqnarray*}
	D\phi(|z|) 	  & = & \varphi' (|z|) \hat z \\
	D^2\phi(|z|) & = & \varphi''(|z|) \hat z \otimes \hat z + \frac{\varphi'(|z|)}{|z|}(I - \hat z \otimes\hat z),
\end{eqnarray*}
and in particular
\[  D^2\phi(\bar p +sz) z \cdot z = 
	\varphi''(|\bar p +sz|)| \widehat{(\overline p +sz)}\cdot z|^2 
	+\frac{\varphi'(|\bar p +sz|)}{|\bar  p+sz|}\left(|z|^2 -| \widehat{(\overline  p +sz)}\cdot z|^2\right) .\]	
Taking into account that $\varphi'' <0$ and $\varphi'>0$, we establish below a lower bound for the first term in the sum above, and an upper bound for the latter term. Take $\rho = |\bar p| \rho_0$ with $\rho_0\in(0,1)$, yet to be determined. Then, for all $z\in B_\rho$ and for all $s\in(-1,1)$, we have 
\[ |\bar p| (1-\rho_0) \leq |\bar p+sz| \leq |\bar p| (1+\rho_0),\]
whereas, for all $z\in\mathcal C_{\eta, \rho}(\bar p) = \{z\in B_\rho; (1-\eta)|z||p|\leq|p\cdot z|\}$ and for all $s\in(-1,1)$,
\[ \left| \left(\bar p + s z\right)\cdot z \right| \geq (1-\eta-\rho_0)|\bar p| |z|. \]
These upper and lower bounds lead to the following estimate 
 \begin{eqnarray*}
  D^2\phi(\bar p +sz) z \cdot z 
  	& \leq & c(\eta,\rho_0)^2\varphi''(|\bar p +sz|) |z|^2 + 
		\left(1-c(\eta,\rho_0)^2\right)\frac{\varphi'(|\bar p +sz|)}{|\bar p+sz|}|z|^2,
\end{eqnarray*}		
with $\displaystyle c(\eta,\rho_0) = (1-\eta-\rho_0)/(1+\rho_0)$. Note that $c(\eta,\rho_0)^2\geq 1-2(\eta+2\rho_0)/(1+\rho_0)\geq 1/2$ for $\eta>0$ and 
$\rho_0>0$ sufficiently small. This implies that
 \begin{eqnarray*}
  D^2\phi(\bar p +sz) z \cdot z 
  	& \leq & \frac12 \varphi''(|\bar p +sz|) |z|^2 + 
		2(\eta+2\rho_0)\frac{\varphi'(|\bar p +sz|)}{|\bar p+sz|}|z|^2,\\
 	& \leq & - \frac14 \frac{ |z|^2 }{ |\bar p+sz| \log^{2}|\bar p+sz|}+ 
		2(\eta+2\rho_0)\frac{|z|^2}{|\bar p+sz|} \\
 	& \leq & - \frac14 \frac{ |z|^2 }{ |\bar p| (1+\rho_0) \log^{2}\left(|\bar p|(1+\rho_0) \right)}+ 
		\frac{ 2(\eta+2\rho_0)|z|^2}{|\bar p|(1-\rho_0)}.
\end{eqnarray*}		
For the choice of constants $\rho_0 = c_1 \log^{-2}(|\bar p|) $ and $\eta = c_2 \log^{-2}(|\bar p|) $, with $c_1,c_2 \in (0, 0.001)$ sufficiently small,  there exists a constant $c>0$, such that, the following estimate holds uniformly for $s\in(-1,1)$,
 \begin{eqnarray*}
  D^2\phi(\bar p +sz) z \cdot z 
 	& \leq & - \frac{1}{64} \frac{ |z|^2 }{ |\bar p| \log^{2}|\bar p|}+ 
		\frac{ (8c_1 + 4c_2)|z|^2}{|\bar p| \log^2(|\bar  p |)} \leq - c  \frac{ |z|^2 }{ |\bar p| \log^{2}|\bar p|}.
\end{eqnarray*}		
Finally, in view of assumption $(K2)$, there exists $C>0$ such that
 \begin{eqnarray*}
 \mathcal T^a[\mathcal C_{\eta,\rho}(\bar p)](\bar x, \bar y) 
 	& \leq &\sup_{a\in\mathcal A} \frac{L}{2} \int_{\mathcal C_{\eta,\rho}(\bar p)}
  	\left(-  \frac{c}{ |\bar p| \log^{2}|\bar p|}\right) |z|^2 \left| K^a(\bar x,z) - K^a(\bar y,z) \right| dz \\
	& \leq & - L c \left( |\bar p| \log^{2}|\bar p|\right)^{-1} 
		C_K \left(c_2  \log^{-2}(|\bar p|)\right)^{\frac{d-1}{2}} c_1 |\bar p|\log^{-2}(|\bar p|)\\
	& \leq & -CL \left( \log^{-2}(|\bar p|)\right)^{\frac{d+3}{2}}.
\end{eqnarray*}			
\end{proof}

The nonlocal kernel is not  bounded in $B$, but it only has a bounded second momentum. Outside the ellipticity cone,  it is necessary to keep the estimate small. In order to obtain an optimal bound for the rest of the terms, we will use a measure decomposition as in \cite{Bci11, Bcci12}, that we briefly discuss next for completeness. Let
\[\Delta K^a(z) : = \Delta K^a(\bar x, \bar y, z)  = K^a(\bar x, z) - K^a(\bar y, z),\]
which is now a changing sign singular kernel. Define $K^{a}_{+},K^{a}_{-}$ as the nonnegative, mutually singular kernel measures satisfying $\Delta K^a = K^{a}_{+} - K^{a}_{-}$ and let  $\Theta^a = \text{supp} (K^{a}_{+})$. Let $K^a_{\min}$ be the minimum of the two kernels, with support $\R^d$. It follows that
\[ K^a(\bar x,z) = K^a_{\min}(z)  + K^{a}_{+}(z) \text{ and } 
   K^a(\bar y,z) = K^a_{\min}(z)  + K^{a}_{-}(z),\]
 where we have dropped the $(\bar x,\bar y)$ dependence on the kernels, to keep the notation short. 
Note that for each pair of appropriate measurable functions $l_1, l_2: \R^d \to \R$ and $D \subset \R^d$ measurable  we can write
\begin{equation}\label{eq:mes-decom}
	\begin{split}
	& \int_{D} l_1(z) K^a(\bar x, z)dz - \int_{D} l_2(z) K^a(\bar y, z)dz \\
	= & \int_{D} \left( l_1(z) - l_2(z)\right) K^a_{\min} (z)dz + \int_{D} l_1(z) K^a_+(z)dz - \int_{D} l_2(z) K^a_-(z)dz.
	\end{split}
\end{equation}

\begin{lemma}[Nonlocal estimate outside the ellipticity cone in $B_\rho$]
Assume $(K3)$ holds with ${\gamma}\in(0,1]$ and let  $\mathcal C_{\eta,\rho}(\bar p)$ as in  $(K2)$, and $\rho\in(0,1)$ be as in Lemma \ref{lemma:estimate-cone}. Then there exists a constant $C>0$ such that, for all $a\in\mathcal A$,
\[\mathcal T^a[B_\rho\setminus \mathcal C_{\eta,\rho}(\bar p)]  (\bar x, \bar y)  
	\leq C L |\bar p|^{{\gamma}} \log^{-2}(|\bar p|) . \]
\end{lemma}

\begin{proof}
Note that, in view of \eqref{eq:Lip-diff-u}, and remark \eqref{eq:mes-decom} above, the nonlocal term outside the ellipticity cone in $B_\rho$ is bounded by
\begin{eqnarray*}
\mathcal T^a[B_\rho\setminus \mathcal C_{\eta,\rho}(\bar p)] (\bar x, \bar y)
	& \leq & L \int_{B_\rho\setminus \mathcal C_{\eta,\rho}(\bar p)}
			\left( \phi(\bar p + z) - \phi(\bar p) -D\phi(\bar p) \cdot z\right)K^a_+(z)dz +  \\
	&       & L \int_{B_\rho\setminus \mathcal C_{\eta,\rho}(\bar p)}
			\left( \phi(\bar p - z) - \phi(\bar p) +D\phi(\bar p) \cdot z\right)K^a_-(z)dz.
\end{eqnarray*}
Using a second-order Taylor expansion of $\phi$ and taking into account that $\varphi$ is smooth, $\varphi'\geq 0$ and $\varphi^{''} \leq 0$, the following bound holds
\begin{eqnarray*}
\mathcal T^a[B_\rho\setminus \mathcal C_{\eta,\rho}(\bar p)]  (\bar x, \bar y)
	& \leq & L \int_{B_\rho\setminus \mathcal C_{\eta,\rho}(\bar p)}
			\sup_{|s|\leq 1} \left(D^2\phi(\bar p + sz) z\cdot z\right) \left(K^{a}_{+}(z) + K^{a}_{-}(z)\right)dz\\
	& \leq & L \int_{B_\rho\setminus \mathcal C_{\eta,\rho}(\bar p)}
			\sup_{|s|\leq 1} \frac{\varphi'(|\bar p+sz|)}{|\bar p+sz|}|z|^2 \left| K^a(\bar x, z) - K^a(\bar y, z)\right| dz.
\end{eqnarray*}
In view of assumption $(K3)$, it follows that there exists $C>0$ such that
\begin{eqnarray*}
\mathcal T^a[B_\rho\setminus \mathcal C_{\eta,\rho}(\bar p)]  (\bar x, \bar y) 
	& \leq & \frac{L}{|\bar p| -\rho} \int_{B_\rho\setminus \mathcal C_{\eta,\rho}(\bar p)}
			|z|^2 \left| K^a(\bar x, z) - K^a(\bar y, z)\right| dz\\ 
	& \leq & \frac{L}{|\bar p| -\rho}  C_K  |\bar p|^{{\gamma}} \rho = 
		     C_K L |\bar p|^{{\gamma}} \frac{c_1 |\bar p| \log^{-2}(|\bar p|)}{|\bar p| \left( 1- c_1\log^{-2}(|\bar p|) \right)}\\
	& \leq & C L|\bar p|^{{\gamma}} \log^{-2}(|\bar p|) .	     
\end{eqnarray*}
\end{proof}

\begin{lemma}[Nonlocal estimate on the circular crown $B\setminus B_\rho$]
Assume $(K3)$ holds with ${\gamma}\in(0,1]$ and let $\rho\in(0,1)$ be as in Lemma \ref{lemma:estimate-cone}. Then there exists a constant $C>0$ such that, for all $a\in\mathcal A$,
\[  \mathcal T^a[B\setminus B_\rho] (\bar x, \bar y) \leq C L |\bar p|^{{\gamma}} \left | \log(|\bar p|) \right|.\]
\end{lemma}

\begin{proof}
As before, in view of \eqref{eq:Lip-diff-u}, and remark \eqref{eq:mes-decom} above, the nonlocal term on the circular crown is bounded by
\begin{eqnarray*}
\mathcal T^a[B\setminus B_\rho] (\bar x, \bar y)
	& \leq & L \int_{B\setminus B_\rho}
			\left( \phi(\bar p + z) - \phi(\bar p) -D\phi(\bar p) \cdot z\right)K^a_+(z)dz +  \\
	&       & L \int_{B\setminus B_\rho}
			\left( \phi(\bar p - z) - \phi(\bar p) +D\phi(\bar p) \cdot z\right)K^a_-(z)dz.
\end{eqnarray*}
Using  the monotonicity, the concavity and the Lipschitz continuity of $\varphi$, the following holds
\begin{eqnarray*}
\mathcal T^a[B\setminus B_\rho] (\bar x, \bar y)
	& \leq & L \int_{B\setminus B_\rho}
			\left( \varphi(|\bar p| + |z|) - \varphi(|\bar p|) +\varphi'(|\bar p|) |\hat p| |z|   \right)
			\left(K^{a}_{+}(z) + K^{a}_{-}(z)\right)dz\\
	& \leq & L \int_{B\setminus B_\rho}
			2 \varphi'(|\bar p|) |z|  \left| K^a(\bar x, z) - K^a(\bar y, z)\right| dz.
\end{eqnarray*}
Employing now the regularity assumption $(K3)$, this further leads to the existence of a  constant $C>0$ so that
\begin{eqnarray*}
\mathcal T^a[B\setminus B_\rho] (\bar x, \bar y)
	& \leq & 2 L \int_{B\setminus B_\rho}
			 |z|  \left| K^a(\bar x, z) - K^a(\bar y, z)\right| dz\\
	& \leq & 2 L \;	C_K |\bar p|^{{\gamma}} \left | \ln \big(c_1 |\bar p| \log^{-2}(|\bar p|)\big) \right|  	 \\
	& \leq & C L |\bar p|^{{\gamma}} \left | \log(|\bar p|) \right|.
\end{eqnarray*}
\end{proof}

It is immediate to see that, in view of the integrability assumption, we have a uniform bound outside the unit ball.

\begin{lemma}[Nonlocal estimate outside the unit ball] 
Assume $(K3)$ holds with ${\gamma}\in(0,1]$. Then there exists a constant $C>0$ such that, for all $a\in \mathcal A$, 
\[\mathcal T^a[B^c](\bar x, \bar y) \leq  C L |\bar p|^{{\gamma}}. \]
\end{lemma}

\begin{proof}
The same measure decomposition as before, gives
\begin{eqnarray*}
\mathcal T^a[B^c](\bar x, \bar y) & \leq & 
	L \int_{B^c} \left(\phi(\bar p - z) - \phi(\bar p) \right) K^a_+(z)dz +
	L \int_{B^c} \left(\phi(\bar p +z) - \phi(\bar p) \right) K^a_-(z)dz\\
	&\leq & 4 L || \phi ||_\infty \left(\int_{B^c} \left| K^a(\bar x, z) - K^a(\bar y, z)\right| dz\right)
	 \leq 4 L C_K  || \phi ||_\infty  |\bar p |^{{\gamma}}.
\end{eqnarray*}
\end{proof}

{\em Step 4. The conclusion}.
Plugging the estimates obtained in the previous lemmas into \eqref{eq:Lip-NLsplit}, we conclude that there exists a universal constant $C>0$, depending only on  the constants given by assumptions $(K1)-(K3)$, such that, for $|\bar p|$ sufficiently small,
\begin{eqnarray*}
\mathcal T^a{(\bar x, \bar y, u)}  
	& \leq &  	 
	-  	C L \left|\log(|\bar p|) \right|^{-(d+3)}  +
		C L |\bar p|^{{\gamma}} \log^{-2}(|\bar p|) +
		C L |\bar p|^{{\gamma}} \left | \log(|\bar p|) \right| + 
		C  L |\bar p|^{{\gamma}}\\ 
	& \leq &  	 
	-  	C L \left(\log^{-2}(|\bar p|) \right)^{\frac{d+3}{2}}  +
		C L |\bar p|^{{\gamma}} |\log(|\bar p|)| +
		C L |\bar p|^{{\gamma}}.
\end{eqnarray*}	
Plugging the above inequality into \eqref{eq:Lip-visc-ineq}, it follows that
\begin{eqnarray*}
	 \delta M_L + C L \left|\log(|\bar p|) \right|^{-(d+3)}  - C L |\bar p|^{{\gamma}} |\log(|\bar p|)| - C L |\bar p|^{{\gamma}} & < &	\nu + C_b L |\bar p|^{{\beta}} + C_f |\bar p|^{{\alpha}}.
\end{eqnarray*}	
Recalling that in view of \eqref{eq:Lip-ineq-M}, ${ |\bar p|\to 0 }$ when $L\to\infty$, and taking into account that for any  $\bar \beta>0$ we have that
$ \displaystyle \lim_{|\bar p|\to 0} \left( |\bar p|^{\bar \beta}  |\log(|\bar p|)|\right)=0, $
it follows that,  up to a modification of the universal constant $C>0$, for sufficiently large $L$,
\begin{eqnarray*}
	 \delta M_L + C L \left|\log(|\bar p|) \right|^{-(d+3)}   & < & \nu + C_f |\bar p|^{{\alpha}}.
\end{eqnarray*}	
Recalling that in view of \eqref{eq:Lip-ineq-M}, $M_L\to 0$ when $L\to\infty$, and $\nu$ can be chosen arbitrarily small, the previous inequality leads to
\begin{eqnarray*}
 	C L \left|\log(|\bar p|) \right|^{-(d+3)}   & \leq & C_f |\bar p|^{{\alpha}}.
\end{eqnarray*}	
In particular, for any $0 < {\sigma} < {\alpha}$, it follows that  $CL |\bar p|^{\sigma} \leq C_f |\bar p|^{{\alpha}}$. Employing further inequality \eqref{eq:Lip-est-Lp} we have $|\bar p|\leq C L^{-1}$, from where the following constraint  holds for $L$, (up to a modification of the universal constant  $C$)
\[ L \leq \frac{C_f}{C} |\bar p|^{{\alpha} - \sigma} \leq 
	\frac{C_f}{C} L^{-{\alpha} + {\sigma}}.\]
Let  $\theta = 1/(1+{\alpha}-{\sigma}) \in(1/(1+{\alpha}),1)$. Choosing then $L > (C_f /C)^{\theta} + 1$, 
we arrive to a contradiction. This concludes the proof.
\end{proof}

\begin{remark}
It is easy to see, from the proof above, that the H\"older continuity of the data can be weakened to a logarithmical modulus of continuity. 
\end{remark}

\begin{remark}\label{rk:reg-coef-exponent}
Notice that, if we assume ${\alpha}=1$, then ${{\sigma}}$ in the statement of the theorem can be chosen arbitrarily close to $1$, and the exponent $1/(1+{\sigma})$ in the Lipschitz bounds is arbitrarily close to $1/2$. This is a crucial estimate to be used in the next section.
\end{remark}

The proof previously developed applies literally to parabolic integro-differential equations. The following holds.
 
\begin{theorem}\label{thm:Lip-ev}
Let  $(f^a)_{a\in\mathcal A}$,  $(b^a)_{a\in\mathcal A}$ two families of bounded functions on $\R^d$ satisfying $(H1)$ with H\"older exponents respectively ${\alpha,\beta}\in(0,1]$ and constants $C_f, C_b$, and $(K^a)_{a\in\mathcal A}$ be a family of kernels satisfying $(K1)-(K3)$ with H\"older exponent $\gamma\in(0,1]$ and constant $C_K$.  Let $u\in BUC(\R^d\times[0,T])$ be a viscosity solution of
\begin{equation*}
	\left\{
	\begin{array}{ll}
	u_t + \mathcal H(x, Du, u)  = 0 	& \text{ in } \R^d\times (0,T]\\
	u(x,0) = u_0(x) 			& \text{ in } \R^d,
	\end{array}
	\right.
\end{equation*}
with $\mathcal H$ is as in \eqref{eq:HJ-d}. If $u_0\in Lip(\R^d)$, then u is Lipschitz continuous with respect to $x$ uniformly on $[0, T ]$, satisfying estimate \eqref{eq:Lip} with a Lipschitz constant depending only on $\alpha$, $\|u\|_\infty$, and on the constants $C_f,C_b,C_K$, but is independent of $\beta,\gamma$.
\end{theorem}

\begin{proof}[Proof of Theorem~\ref{thm:Lip-ev}] 
We proceed similarly to the proof of Theorem ~\ref{thm:Lip}, with the following function which doubles the variables
\[ \Phi(x,y,t,s) = u(x,t) - u(y,s) - L\phi(x-y) -C|t-s|- \psi_\zeta(x), \]
where $C>0$ is a constant and $\phi$ is defined as in the proof of Theorem ~\ref{thm:Lip}. 	
The previous proof literally adapts to the  parabolic case, since  the non linearity $\mathcal H$ is independent of time.  \end{proof}

\section{The cell problem and the effective Hamiltonian}\label{sec:eff-H}

In this section we establish the well-posedness of the cell problem and give a {\em fine} Lipschitz regularity estimate for  the corrector, that will later play a crucial role in the proof of convergence. Further, we set forth a series of properties for the effective Hamiltonian, which shall have an implicit nonlocal dependence on the the averaged profile.

\subsection{The cell problem.} 

As made precise in Section \ref{sec:prelim},  the cell problem both in the symmetric and the non-symmetric case can be  formulated as follows.  Given $x, p \in \R^d$ and a function $u\in \mathcal C^2(\R^d)\cap L^\infty(\R^d)$ show that there exists a unique constant $\lambda \in \R$ so that the following problem has a periodic, continuous viscosity solution
\begin{equation}\label{eq:cell}
	\sup_{a\in \mathcal A} \{ - {\mathcal I^a(\xi,\psi)} - \tilde b^a(\xi;x) \cdot D\psi(\xi) - \tilde f^a(\xi;x,p,u) \} = \lambda \quad \text{ in } \R^d,
\end{equation}
where the source term is given by 
\[ \tilde f^a(\xi;x,p,u)  =  f^a(x,\xi) + b^a(x,\xi) \cdot p + \mathcal L^a(x, \xi, u), \]
with $\mathcal L^a$ defined by \eqref{eq:NL-L}. However, the nonlocal operator  $ \mathcal I^a{(\xi,\psi)} $ and the drift term $ \tilde b^a$ are defined differently according to the symmetry of the nonlocal kernel. 
\begin{enumerate}
\item In the case of symmetric kernels - assumption $(Ks)$,  the nonlocal operator is given by
	\[ \mathcal I^a{(\xi,\psi)} =
	\int_{\R^d} \big( \psi(\xi+ z) - \psi(\xi) - \mathbf{1}_B(z) D\psi (\xi) \cdot z \big) K^a(\xi, z)dz, \]  
	and the drift is $ \tilde b^a(\xi; x)  =  b^a(x,\xi).$
\item In the non-symmetric case - assumption $(Kns)$, the nonlocal operator is just 
	\[ \mathcal I^a{(\xi,\psi)} = -k^a(\xi,0)(-\Delta)^{1/2}\psi(\xi) \]
	whereas the drift adds an extra term
	$ \tilde b^a(\xi; x)  =  b^a(x,\xi) + b_K^a(\xi),$
	with $b^a_K:\R^d\to\R^d$ given by
	\[b^a_K(\xi) = \int_B \left(k^a(\xi,z) - k^a(\xi,0)\right)\frac{z}{|z|^{d+1}} dz.\]
\end{enumerate}
In what follows, proofs are nowhere different in the symmetric or the non-symmetric case. This explains why we want to keep everything under a unified notation.

The well-posedness of problem \eqref{eq:cell} is standard \cite{Lpv86,Bcci12, Bct19}, except for few arguments due to the lack of comparison. We show that the corrector is Lipschitz continuous and give in addition a {\em fine estimate} for the Lipschitz constant. This estimate plays a central role in establishing a comparison principle for the effective equation, which in turn will be helpful in establishing homogenization.

\begin{theorem}\label{thm:cell}
Let  $(f^a)_{a\in\mathcal A}$ and $(b^a)_{a\in\mathcal A}$ be  two families of bounded functions on $\R^{2d}$, satisfying $(H0)$, $(H1)$ with respect to the fast variable $\xi$ and with H\"older exponents respectively $\alpha,\beta \in (0,1]$. Let $(K^a)_{a\in\mathcal A}$ be a family of kernels satisfying $(K0)-(K3)$ with H\"older exponent $\gamma\in(1/2,1]$. 
Then, for any $x, p \in \R^d$ and ${u\in\mathcal C^2(B_\rho(x))\cap L^\infty(\R^d)}$ for some $\rho\in(0,1]$,  there exists a unique constant $ \lambda\in \R$ so that problem \eqref{eq:cell} has a Lipschitz continuous, periodic viscosity solution $\psi$. Moreover, $\psi$ satisfies the following Lipschitz bound: there exists  $\sigma\in(0,\min( \alpha,\beta,\gamma ))$ such that,  for all $\xi_1,\xi_2\in\R^d$,
\begin{equation}\label{eq:cell-Lip-corr}
	|\psi(\xi_1) - \psi(\xi_2)| \leq C_\sigma(1 + |p| + {C_\rho^{x,u}})^{\frac{1}{1+\sigma}} |\xi_1 - \xi_2|,
\end{equation}
where $C_\sigma>0$ is a constant depending on $\alpha$, $||\psi||_\infty$, and ${C_\rho^{x,u}}$ is given by
\begin{equation}\label{eq:cell-Lip-corr-const}
 	{C_\rho^{x,u}}:=   || D^2u ||_{ L^\infty(B_\rho(x))}   \rho +   |D u(x) | |\ln(\rho)| + ||u||_{\infty} \rho^{-1}.
\end{equation}
\end{theorem}

\begin{remark} 
In the case of symmetric kernels, the compensator is not needed and the constant  writes
 \begin{equation*}
 	{C_\rho^{x,u}}:=   || D^2u ||_{ L^\infty(B_\rho(x))}   \rho  + ||u||_{\infty} \rho^{-1}.
\end{equation*}
\end{remark}

\begin{proof}[Proof of Theorem~\ref{thm:cell}]
In view of the available regularity estimates, we rely on  a new comparison principle for general L\'evy measures, shown in Proposition \ref{prop:comp} of the Appendix. Then, the proof follows the same arguments as for instance in~\cite{Bcci14, Bklt15}, where measures were of L\'evy-It\^o type and comparison was for free (see \cite{Bi08}). We provide here the main ideas of the proof.

Fix $x,p\in\R^d$ and ${u\in\mathcal C^2(B_\rho(x))\cap L^\infty(\R^d)}$ with $\rho\in(0,1]$. Let $\delta>0$ and consider the approximated problem
\begin{equation}\label{eq:cell-approx}
\delta {\psi^\delta}+ \sup_{a\in \mathcal A} \{- {\mathcal I^a(\xi,\psi^\delta)} - 
	\tilde b^a(\xi;x) \cdot D\psi^\delta(\xi) - \tilde f^a(\xi; x,p,u) \} = 0.
\end{equation}

\begin{lemma}\label{lem:cell-approx}
There exists a Lipschitz continuous viscosity solution $\psi^\delta$ of problem \eqref{eq:cell-approx}.
\end{lemma}

\begin{proof}[Proof of Lemma \ref{lem:cell-approx}]
We use a vanishing-coercivity argument in ordr to establish th existence of a uniformly continuous solution. More precisely, for any $\eta>0$, consider the coercive problem 
\begin{equation}\label{eq:cell-approx}
\delta {\psi^{\delta,\eta}}+ \sup_{a\in \mathcal A} \{- {\mathcal I^a(\xi,\psi^{\delta,\eta})} - 
	\tilde b^a(\xi;x) \cdot D\psi^{\delta,\eta}(\xi) - \tilde f^a(\xi; x,p,u) \} + \eta |D\psi^{\delta,\eta}|^2= 0,
\end{equation}
which in view of the results of \cite{Bklt15} admits a H\"older continuous viscosity solution. In view of Theorem \ref{thm:Lip} the solutions are Lipschitz continuous, with a Lipschitz norm independent of $\eta$. Indeed, in order to cope with the quadratic (but autonomous) gradient term, one should look at the approximated equation with $|D\psi^{\delta,\eta}|$ replaced by $\max(|D\psi^{\delta,\eta}|,R)$, for $R>0$, and remark that its solutions are Lipschitz continuous, with the Lipschitz norm independent of $R$. Moreover, if we denote  $M = \sup_{a\in\mathcal A } || \tilde f^a||_\infty$, we note that $ ||\psi^{\delta,\eta}||_\infty\leq M/\delta$. Thus, passing to the limit, it follows that there exists a Lipschitz continuous solution of \eqref{eq:cell-approx} which satisfies  $||\psi^\delta||_\infty\leq M/\delta. $
\end{proof}

Consider  the sequence of functions 
\[ \tilde \psi^\delta (\xi) : = \psi^\delta(\xi) - \psi^\delta(0),\]
which satisfy the equation
\begin{equation*}
	\delta  \tilde \psi^\delta + \sup_{a\in \mathcal A} \{- 	{\mathcal I^a(\xi,\tilde \psi^\delta)} - 
	\tilde b^a(\xi;x) \cdot D\tilde\psi^\delta(\xi) - \tilde f^a(\xi; x,p,u) \} = - \delta\psi^\delta(0).
\end{equation*}
In view of the strong maximum principle (see \cite{C12}), it can be shown as in \cite{Bcci14} that the above family of functions is precompact. Indeed, the following holds.

\begin{lemma}\label{lem:conv}
The sequence $\{\tilde \psi^\delta (\cdot)\}_\delta$ is uniformly bounded and uniformly Lipschitz continuous.
\end{lemma}

\begin{proof}[Proof of Lemma \ref{lem:conv}]
We argue by contradiction and assume there exists a subsequence for which the associated sequence of norms blows up, i.e. $ ||\tilde \psi^\delta||_\infty \to\infty$, as $\delta\to 0$. Consider the renormalized functions
\[ \hat \psi^\delta(\xi) = \frac{\tilde \psi^\delta (\xi)}{||\tilde \psi^\delta||_\infty},\]
which satisfy the equation
\begin{equation*}
	\delta  \hat \psi^\delta + \sup_{a\in \mathcal A} \left\{- {\mathcal I^a(\xi,\hat \psi^\delta)} - 
	\tilde b^a(\xi;x) \cdot D\hat\psi^\delta(\xi) -  \frac{\tilde f^a(\xi; x,p,u)}{||\tilde \psi^\delta||_\infty} \right\} 
	= - \frac{\delta \psi^\delta(0)}{||\tilde \psi^\delta||_\infty}.
\end{equation*}
Since the renormalized functions all have norm $||\hat \psi^\delta||_\infty = 1$, it follows from Theorem  \ref{thm:Lip} that the family is equi-Lipschitz continuous. Thus, by the Ascoli-Arzela theorem, there exists a subsequence of periodic functions $\{\hat \psi^{\delta_n} (\cdot)\}_{\delta_n}$ which converges locally uniformly - and globally in view of the periodicity -, to a function $\hat \psi$ satisfying the equation
\begin{equation*}
 	\sup_{a\in \mathcal A} \left\{- {\mathcal I^a(\xi,\hat \psi)} - 
	\tilde b^a(\xi;x) \cdot D\hat\psi(\xi) \right\} = 0.
\end{equation*}
The latter equation satisfies the strong maximum principle (see \cite{C12}), while its solution has $||\hat\psi||_\infty =1$ and $\hat\psi(0) = 0$, which leads to a contradiction. Thus, the sequence of functions $\{\tilde\psi^{\delta_n} (\cdot)\}_{\delta_n}$  is uniformly bounded. 
In view of Theorem \ref{thm:Lip}, the family is also uniformly Lipschitz continuous.
\end{proof}

In view of  Ascoli-Arzela theorem, there exists a subsequence $\big(\tilde \psi^{\delta_n}\big)_{\delta_n}$ which converges locally uniformly (and globally due to periodicity) to a periodic, Lipschitz continuous function 
\[ \psi = \lim_{\delta_n \to 0} \psi^{\delta_n}. \] 
Moreover $\big(\delta_n \psi^{\delta_n}(0)\big)_{\delta_n}$ is bounded and, up to a subsequence, there exists a constant $\lambda\in\R$, so that 
\[ \lambda  = -\lim_{\delta_n \to 0} \delta \psi^{\delta_n}(0).\]
The uniqueness of the constant $\lambda$ follows from the comparison principle stated in Proposition \ref{prop:comp}.


Furthermore, in view of Theorem \ref{thm:Lip}, we obtain the following Lipschitz estimate for the  corrector. In view of  $(K3)$, there exists a  constant $C>0$ such that, for any $a\in\mathcal A$, and for all $\xi_1,\xi_2\in\R^d$,
\begin{eqnarray*}
 	\big | \mathcal L^a(x, \xi_1,u) - \mathcal L^a(x,\xi_2,u) \big| 
	&\leq &	\| D^2u \|_{ L^\infty(B_\rho(x))} \int_{B_\rho}    |z|^2 \big| K^a(\xi_1, z) - K(\xi_2,z) \big| dz  + \\
	& &  		| Du(x) | \int_{B\setminus B_\rho}   |z| \big| K^a(\xi_1, z) - K(\xi_2,z) \big| dz  + \\
	& &  		2 \| u \|_{L^\infty (B_\rho^c(x))} \int_{\R^d\setminus B_\rho} \big| K^a(\xi_1, z) - K(\xi_2,z) \big| dz	\\
 	& \leq  & 	C_K \Big( ||D^2 u||_{L^\infty(B_\rho(x))} \rho +  
			|Du(x)| |\ln(\rho)|  + 
			2 ||u||_{\infty} \rho^{-1} \Big) |\xi_1-\xi_2|^{\gamma} \\ \smallskip
	&\leq & 2C_K\; {C_\rho^{x,u}} \; | \xi_1-\xi_2|^{\gamma},
\end{eqnarray*}
where ${C_\rho^{x,u}}$ is given by \eqref{eq:cell-Lip-corr-const}. In view of assumption $(H1)$ it follows that, for any $a\in\mathcal A$, and for all $\xi_1,\xi_2\in\R^d$,
\begin{eqnarray*}
	\left| \tilde f^a(\xi_1; x, p,u) - \tilde f^a(\xi_1; x, p,u) \right|
	&\le& \left( C_f |\xi_1-\xi_2|^\alpha + C_b |p| |\xi_1-\xi_2|^\beta + 2C_K\;C^{u,x}_{\rho}|\xi_1-\xi_2|^\gamma \right) \\
	&\le& \max(C_b,C_f,2C_K) \left( 1 +  |p| + {C_\rho^{x,u}}\right) |\xi_1-\xi_2|^{\min( \alpha,\beta,\gamma )}. 
\end{eqnarray*}	
Thus, $\tilde f^a$ is H\"older continuous in $\xi$, with H\"older coefficient $\tilde \alpha = \min(\alpha,\beta,\gamma)$. In view of Theorem \ref{thm:Lip},  we conclude that for each $\sigma \in (0,\min( \alpha,\beta,\gamma ))$, there exists $C_\sigma >0$ depending on $||\psi||_\infty$  such that, for all $\xi_1,\xi_2\in\R^d$, it holds
\begin{equation*}\label{eq:cell-Lip-corr}
	|\psi(\xi_1) - \psi(\xi_2)| \leq C_\sigma \left (1 + |p| + {C_\rho^{x,u}}\right)^{\frac{1}{1+\sigma}} |\xi_1 - \xi_2|.
\end{equation*}
\end{proof}

\begin{remark}
The Lipschitz estimate \eqref{eq:cell-Lip-corr} holds  for the approximate corrector $\psi^\delta$ as well.
\end{remark}
\subsection{The effective Hamiltonian}

The ergodic constant in Theorem  \ref{thm:cell} has a local dependence on $x,p\in\R^d$, and a nonlocal dependence with respect to $u\in\mathcal C^2(\R^d)\cap L^\infty(\R^d)$. To display explicitly this dependence, we hereafter write
\[ \lambda = {\overline H}(x,p,u),\] 
and call ${\overline H}$ the {\em effective Hamiltonian}, which is well defined as a global function 
\[ \overline H : \R^d\times\R^d\times \mathcal C^2(\R^d)\cap L^\infty(\R^d) \to \R.\] 

\begin{remark}
In fact, in view of Theorem \ref{thm:cell}, for fixed $(x,p)\in\R^{2d}$, the effective Hamiltonian is well defined for functions which are only in $\mathcal C^2(B_\rho(x))\cap L^\infty(\R^d) = : \mathcal E_\rho^x$, for some $\rho\in(0,1]$.  Denote $\mathcal E^x = \bigcup_{\rho>0} \mathcal E_\rho ^x$ and  introduce the space 
\[\mathcal E:= \left\{(x,u)\in\R^d\times L^{\infty}(\R^d): 
	\text{ there exists } \rho >0 \text{ s.t. } u\in\mathcal C^2(B_\rho(x))\right\}\] 
One could consider $\overline H$ as a function
\begin{eqnarray*}
 	&& \overline  H :  \R^d\times\mathcal E \to\R \\
 	&& \overline H(p, (x,u)) = \lambda.
\end{eqnarray*}
This turns out to be useful when viscosity solutions associated to the effective Hamiltonian are employed. Similar to viscosity solutions associated to the original problem \eqref{eq:HJ-eps}, or its stationary variant, when dealing with the nonlocal term it is often convenient to replace $\mathcal C^2(\R^d)$ test functions $\phi$ by their local truncation  around $x$ in a small neighbourhood, namely by $\1_{B_\rho(x)}\phi+ \1_{B_\rho(x)}u$. However, since the nonlocal dependence of the effective Hamiltonian is not explicit, we will not to be able to give (later on) equivalent definitions of viscosity solutions in terms of smooth or less regular test functions. In this sense, it is crucial for $\overline H$ to make sense for locally $\mathcal C^2(B_\rho(x))$ functions.
\end{remark}

\begin{remark}
Note in addition that,  for fixed $p\in\R^d$, one can write $\overline H$ as a function
\begin{eqnarray*}
 	&& \overline H_p : \mathcal C^2(\R^d)\cap L^\infty(\R^d) \to \mathcal F(\R^d)\\
 	&& \overline H_p[u](x) = \overline H(x,p,u),
\end{eqnarray*}
 where  $\mathcal F(\R^d)$ is the space of all functions $h:\R^d\to\R$.  
 \end{remark}
 
 We will see below that in fact $\overline H$ maps $\mathcal C^2$ functions into continuous functions, is convex in $u$ and in $p$, and it satisfies a global comparison principle. We will use the space $\mathcal C^{2,\sigma}(\R^d)$ to be the collection of functions $u$, with continuous second derivatives on $\R^d$ with $|| u ||_{\mathcal C^{0,\sigma}(\R^d)}$, $|| Du ||_{\mathcal C^{0,\sigma}(\R^d)}$, $|| D^2u ||_{\mathcal C^{0,\sigma}(\R^d)}$ all finite. More precisely, the following structural properties hold for $\overline H$.

\begin{prop}\label{prop:eff-H}
Let  $(f^a)_{a\in\mathcal A}$ and $(b^a)_{a\in\mathcal A}$ be  two families of bounded functions on $\R^{2d}$, satisfying $(H0)$ and $(H1)$ with respect to both variables with $\tilde{d}=2d$ and  with H\"older exponents respectively $\alpha,\beta \in (0,1]$. Let $(K^a)_{a\in\mathcal A}$ be a family of kernels satisfying $(K0)-(K3)$ with H\"older exponent $\gamma\in(1/2,1]$.  
Then, the effective Hamiltonian satisfies the following properties.
\begin{enumerate}
\item Fix $(x,p)\in\R^d$ and let $u_1,u_2\in\mathcal E^x$. Then
	\[ {\overline H}(x,p,u_1) - {\overline H}(x,p,u_2) \geq 
	   - \sup_{a\in\mathcal A} \sup_{\xi\in\R ^d}\big(\mathcal L^a(x,\xi,u_1) - \mathcal L^a(x,\xi,u_2)\big).\]
	In particular, $\overline H$ satisfies {\em the global comparison principle} :
	if $u_1,u_2\in\mathcal E^x$ such that $u_1\leq u_2$ in $\R^d$ and $u_1(x)=u_2(x)$,
	then ${\overline H}(x,p,u_1) \ge {\overline H}(x,p,u_2).$ \smallskip
	
\item For any $(x,p)\in\R^{2d}$, $\overline H(x,p,\cdot)$ is convex, i.e. for any
 	$u_1,u_2\in\mathcal E^x$ and $s\in(0,1)$, 
	\[\overline H(x,p, s u_1 + (1-s)u_2) \leq 
	  s \overline H(x,p,u_1) + (1-s) \overline H(x,p,u_2).\]
	  
\item There exists a constant $B >0$  such that for all  $x\in\R^d$, 
	$u\in\mathcal E^x$ and  $p_1,p_2\in\R^d$,
	\[ \left| {\overline H}(x,p_1,u) - {\overline H}(x,p_2,u))\right| \leq B \left| p_1-p_2\right|.\]
	
\item Fix  $p\in\R^d$. Then ${\overline H}_p : C^{2,\sigma}(\R^d)\cap L^\infty(\R^d) \to C^{0,\sigma}(\R^d)$, for any 
	$\sigma\in(0,\min(\alpha,\beta,\gamma))$, i.e. for any $u\in C^{2,\sigma}(\R^d)\cap L^\infty(\R^d)$ 
	there exists a constant $C = C(p,u) >0$ such that, for all $x_1,x_2\in\R^d$, 
	\[ \left| {\overline H}(x_1,p,u) - {\overline H}(x_2,p,u)\right| \leq C |x_1-x_2|^\sigma.\]	
\end{enumerate}
\end{prop}
\begin{remark}
In most cases,  little can said about the nonlocal structure of the nonlocal operator. It is known for instance, that if a nonlocal operator satisfies the global maximum principle,  is {\em linear} and maps $C^2(\R^d)$ into $C(\R^d)$, then it takes the Courr\`ege form (see Theorem 1.5 in \cite{Co65}). In our setup, the mapping of $\overline H$ from $C^{2,\sigma}(\R^d) $ to $C^{0,\sigma}(\R^d)$ is convex, so it is natural to expect that $\overline H$ takes the Bellman form over the Courr\'ege operators. However, no rigorous result is proven in this respect.
\end{remark} 

\begin{proof} We  show each of these points separately, though a global argument could be applied.

(1) Fix $(x,p)\in\R^d\times\R^d$ and let $u_1,u_2\in\mathcal E^x$.  
Consider the triplets $(x,p,u_1)$ and $(x,p,u_2)$ and denote  their corresponding approximate correctors $\psi_1^\delta$ and ${\psi_2^\delta}$, which solve the equations
\begin{eqnarray*}
	\delta {\psi_1^\delta}(\xi) + \sup_{a\in \mathcal A} \{- \mathcal I^a(\xi,{\psi_1^\delta}) - 
	\tilde b^a(\xi;x) \cdot D{\psi_1^\delta}(\xi) - \tilde f^a(\xi; x,p,u_1) \}  & = & 0, \\
	\delta {\psi_2^\delta}(\xi) + \sup_{a\in \mathcal A} \{- \mathcal I^a(\xi,{\psi_2^\delta}) - 
	\tilde b^a(\xi;x) \cdot D{\psi_2^\delta}(\xi) - \tilde f^a(\xi; x,p,u_2) \}  & = & 0.
\end{eqnarray*}
It is easy to see that ${\psi_1^\delta}$ is a viscosity subsolution for 
\begin{eqnarray*}
	\delta {\psi_1^\delta} (\xi)
	& + & \sup_{a\in \mathcal A} \{- \mathcal I^a(\xi,{\psi_1^\delta}) - 
		 \tilde b^a(\xi; x) \cdot D{\psi_1^\delta}(\xi) - \tilde f^a(\xi; x,p,u_2) \}  \\
	&\le& \sup_{a\in\mathcal A} \left(\mathcal L^a(x,\xi,u_1) - \mathcal L^a(x,\xi,u_2)\right).
\end{eqnarray*}
Taking into account that $x$ is a local maximum of $u_1-u_2$ and that $u_1,u_2\in\mathcal C^2(B_\rho(x))$,  it follows that $Du_1(x) = Du_2(x)$ and thus, for all $a\in\mathcal A$, 
\[\mathcal L^a(x,\xi,u_1) - \mathcal L^a(x,\xi,u_2) \leq 0. \]
Therefore,  $\psi_1^\delta$ is a viscosity subsolution of 
\begin{eqnarray*}
	\delta {\psi_1^\delta}(\xi) + 
	\sup_{a\in \mathcal A} \{- \mathcal I^a(\xi,{\psi_1^\delta}) - 
	\tilde b^a(\xi; x) \cdot D{\psi_1^\delta}(\xi) - \tilde f^a(\xi; x,p,u_2) \}  \leq 0.
\end{eqnarray*}
Since the approximate correctors are Lipschitz, it follows from the comparison principle for Lipschitz functions given in Proposition \ref{prop:comp}, that ${\psi_1^\delta}\leq {\psi_2^\delta}$ in $\R^d$, which further leads to
\[ {\overline H}(x,p,u_1) = 
   - \lim_{\delta \to 0} \delta {\psi_1^\delta}(0) \geq
   - \lim_{\delta \to 0} \delta {\psi_2^\delta}(0)  = {\overline H}(x,p,u_2).\]

(2) In order to prove convexity, under the same notations as above, consider as well for any $s\in(0,1)$ the triplet $(x,p,(1-s)u_1+s u_2)$ and its approximate corrector $\psi_{s}^\delta$, which solves the equation
\[\delta \psi_{s} ^\delta(\xi) + \sup_{a\in \mathcal A} \{- \mathcal I^a(\xi,\psi_{s}^\delta) - 
  \tilde b^a(\xi;x) \cdot D\psi_s^\delta(\xi) - \tilde f^a(\xi; x,p,(1-s)u_1+su_2) \}   =  0.\]
It is standard to check that $(1-s)\delta \psi_1^\delta+s\delta\psi_2^\delta $ is a viscosity subsolution of the above equation. In view of the comparison principle given in Proposition \ref{prop:comp}, it follows that
\[ (1-s)\delta \psi_1^\delta+s\delta\psi_2^\delta\leq\delta \psi_s^\delta, \]
which implies, as  $\delta \to 0$, the convexity of $\overline H$ with respect to $u$. \smallskip

(3) Fix $x\in\R^d$, let $u\in\mathcal E^x$ and $p_1,p_2\in\R^d$. Denote by $\psi_1^\delta$ and $\psi_2^\delta$ the approximate correctors corresponding to $p_1$ and $p_2$, viscosity solutions of
\begin{eqnarray*}
	\delta \psi_1^\delta(\xi) + \sup_{a\in \mathcal A} \{- \mathcal I^a(\xi,\psi_1^\delta) - 
	\tilde b^a(\xi; x) \cdot D\psi_1^\delta(\xi) - \tilde f^a(\xi); x,p_1,u) \}  & = & 0, \\
	\delta \psi_2^\delta(\xi) + \sup_{a\in \mathcal A} \{- \mathcal I^a(\xi,\psi_2^\delta) - 
	\tilde b^a(\xi; x) \cdot D\psi_2^\delta(\xi) - \tilde f^a(\xi; x,p_2,u) \}  & = & 0.
\end{eqnarray*}
Then $\psi_1$ solves in the viscosity sense
\begin{eqnarray*}
	\delta \psi_1^\delta(\xi) 
	& + & \sup_{a\in \mathcal A} \{- \mathcal I^a(\xi,\psi_1^\delta) - 
	\tilde b^a(\xi; x) \cdot D\psi_1^\delta(\xi) - \tilde f^a(\xi; x,p_2,u) \} \\
	 &\le& \sup_{a\in \mathcal A} \left| \tilde f^a(\xi; x,p_1,u)- \tilde f^a(\xi; x,p_2,u) \right| 
	 \leq \sup_{a\in \mathcal A} || b^a ||_\infty |p_1-p_2|.
\end{eqnarray*}	 
In view of the comparison principle given in Proposition \ref{prop:comp}, it follows that, for $B = \sup_{a\in \mathcal A} || b^a ||_\infty$,
\[  \delta \psi_1^\delta \leq \delta \psi_2^\delta + B|p_1-p_2|. \]
Reverting $p_1$ and $p_2$ we get the bound from below. Letting $\delta\to 0$, the conclusion follows.\smallskip

(4) Fix $p\in\R^d$, let $u\in\mathcal C^{2,\sigma}(\R^d)\cap L^\infty(\R^d)$ and take $x_1,x_2\in\R^d$. Let $\psi_1^\delta$ and $\psi_2^\delta$ be the approximate correctors corresponding to $x_1$ and $x_2$, thus viscosity solutions of
\begin{eqnarray*}
	\delta \psi_1^\delta(\xi) + \sup_{a\in \mathcal A} \{- \mathcal I^a(\xi,\psi_1^\delta) - 
	\tilde b^a(\xi; x_1) \cdot D\psi_1^\delta(\xi) - \tilde f^a(\xi; x_1,p,u) \}  & = & 0, \\
	\delta \psi_2^\delta(\xi) + \sup_{a\in \mathcal A} \{- \mathcal I^a(\xi,\psi_2^\delta) - 
	\tilde b^a(\xi; x_2) \cdot D\psi_2^\delta(\xi) - \tilde f^a(\xi; x_2,p,u) \}  & = & 0.
\end{eqnarray*}
Then $\psi_1^\delta$ is a viscosity subsolution of 
\begin{eqnarray*}
	\delta \psi_1^\delta(\xi)
	& + & \sup_{a\in \mathcal A} \{- \mathcal I^a(\xi,\psi_1^\delta) - 
	b^a(\xi;x_2) \cdot D\psi_1^\delta(y) - \tilde f^a(\xi;x_2,u,p) \}  \\
	& \leq &  \sup_{a\in\mathcal A} \left| b^a(\xi;x_1) - b^a(\xi;x_2)\right| || D\psi_1^\delta||_\infty + 
	|\tilde f^a(\xi; x_1,p,u) - \tilde f^a(\xi; x_2,p,u) |.
\end{eqnarray*}
In view of assumption $(K1)$ the following holds, uniformly in $\xi,p\in\R^d$ and for all $a\in\mathcal A$,
\begin{eqnarray*}
	\big | \mathcal L^a(x_1,\xi,u) & - & \mathcal L^a(x_2,\xi,u)  \big | \\
	&\leq & \int_0^1(1-t) dt \int_{B_\rho} \big| D^2u(x_1+tz)-D^2u(x_2+tz) \big| |z|^2 \; |K^a(\xi, z)|dz +\\
	&&      \int_0^1 t dt \int_{B\setminus B_\rho} \big| Du(x_1+tz)-Du(x_2+tz) \big| |z|\; |K^a(\xi, z)|dz +\\
	&&      \int_{B\setminus B_\rho} \big| Du(x_1) -Du(x_2) \big| |z| \; |K^a(\xi, z)|dz + \\
	&&     \int_{B^c} \big| (u(x_1 + z) - u(x_2 + z)) - (u(x_1) - u(x_2))\big| \; |K^a(\xi, z)|dz \\		
	&\leq &   ||D^2u||_{C^{0,\sigma}(B(x_1))} |x_1-x_2|^\sigma \int_{B_\rho}  |z|^2 \; |K^a(\xi,z)|dz + \\  
	&&  2 ||Du||_{C^{0,\sigma}(B(x_1))} |x_1-x_2|^\sigma \int_{B\setminus B_\rho}   |z| \; |K^a(\xi,z)|dz + \\  
	&&  2 || u ||_{C^{0,\sigma}(\R^d)}  |x_1-x_2|^\sigma\int_{B^c}   |K^a(\xi,z)| dz\\
	&\leq & C || u ||_{C^{2,\sigma}(\R^d)}|x_1-x_2|^\sigma,
\end{eqnarray*}
where $C$ is a universal constant. In view of the regularity assumption $(H1)$, the previous inequality leads to 
\begin{eqnarray*}
|\tilde f^a(\xi; x_1,p,u) - \tilde f^a(\xi; x_2,p,u) | 
	& \leq & C_f |x_1-x_2|^\alpha + C_b |x_1-x_2|^\beta |p| + C || u ||_{C^{2,\sigma}} |x_1-x_2|^\sigma\\
	& \leq & C\Big(1+|p| +  || u ||_{C^{2,\sigma}(\R^d)}\Big) |x_1-x_2|^\sigma.
\end{eqnarray*}
The Lipschitz regularity of the approximate corrector, implies that, for any $\rho\in(0,1)$,
\[ || D\psi_1^\delta||_\infty \leq  C \left (1 + |p| + {C_\rho^{x_1,u}}\right)^{\frac{1}{1+\sigma}}.\]
In particular for $\rho=1$, we have ${C_1^{x_1, u}} \leq || u ||_{C^{2}(\R^d)}$ and hence there exits $C>0$ so that
\[ || D\psi_1^\delta||_\infty \leq  C_\sigma \left (1 + |p| + || u ||_{C^{2}(\R^d)}\right)^{\frac{1}{1+\sigma}}.\]
Therefore, we conclude that $\psi_1^\delta$ is a viscosity subsolution, in $\R^d$, of 
\begin{eqnarray*}
	\delta \psi_1^\delta(\xi)
	& + & \sup_{a\in \mathcal A} \{- \mathcal I^a(\xi,\psi_1^\delta) - 
	b^a(\xi;x_2) \cdot D\psi_1^\delta(\xi) - \tilde f^a(\xi;x_2,u,p) \}  \\
	& \leq & C \Big( \left (1 + |p| + || u ||_{C^{2}(\R^d)}\right)^{\frac{1}{1+\sigma}} + 
	\big(1+|p| +  || u ||_{C^{2,\sigma}(\R^d)}\big) \Big) |x_1-x_2|^\sigma,  
\end{eqnarray*}
up to a modification of the universal constant $C$. In view of the comparison principle for Lipschitz functions, given in Proposition \ref{prop:comp}, it follows that there exists a constant 
\[ C(p,u) : =C \Big( \left (1 + |p| + || u ||_{C^{2}(\R^d)}\right)^{\frac{1}{1+\sigma}} + 
		\big(1+|p| +  || u ||_{C^{2,\sigma}(\R^d)}\big) \Big)\]
such that, uniformly in $\delta$ and $\xi$,
\[ \delta \psi_1^\delta (\xi) - \delta \psi_2^\delta (\xi) \leq C(p,u)|x_1-x_2|^\sigma.\]
Reverting $x_1$ and $x_2$ we get the lower bound. Thus, letting $\delta\to 0$, the conclusion follows.
\end{proof}

We give in the following corollary  the global behaviour of $\overline H$ with respect to all of its variables and give an ellipticity growth condition. This turns out to be  fundamental in order to perform later on a linearization for the effective problem (see the following section).  The result  strongly relies on the Lipschitz estimate of the solution to the cell problem \eqref{eq:cell} given by Theorem \ref{thm:cell}.

\begin{corollary}\label{cor:eff-H}
Let the same assumptions as in Proposition \ref{prop:eff-H} hold. For any $x_1,x_2,p_1,p_2\in\R^d$, $u_1\in\mathcal E_\rho^{x_1}$ and  $u_2\in\mathcal E_{\rho}^{x_2}$ with $\rho>0$, the  following holds, for any $\sigma \in (0,\min(\alpha,\beta,\gamma))$,
\begin{eqnarray*}
\overline H(x_2,p_2,u_2) -  \overline H(x_1,p_1,u_1)  & \leq &  C
	\Big( \left (1 + |p_1| + {C_\rho^{x_1, u_1}} \right)^{\frac{1}{1+\sigma}} + 
	\big(1+|p_1| \big) \Big) |x_1-x_2|^{{\min(\alpha,\beta)}}\\
	&& + \sup_{a\in\mathcal A} || b^a ||_ {\infty} |p_1-p_2| + 
	\sup_{\substack{a\in\mathcal A\\ \xi\in\Pi^d}} \big(- \mathcal L^a(x_2,\xi,u_2) + \mathcal L^a(x_1,\xi,u_1) \big),
\end{eqnarray*}
where
$C_\rho^{x_1, u_1} = || D^2u_1 ||_{L^\infty(B_\rho(x_1))} \rho + |Du_1(x_1)|| \; |\ln(\rho)| + || u_1||_{\infty} \rho ^{-1}.$
\end{corollary}

\begin{proof}
It is easy to see from the previous proof that, the following improved estimate holds for the global variables. This is due to the fact that we drop the estimate of the nonlocal terms $\mathcal L^a$ which appear in the definition of $\tilde f^a$. Indeed, the $C^{2,\sigma}$ norm of $u$ appearing in the computation at the end of the proof of Proposition \ref{prop:eff-H} and stemming from the estimate of the nonlocal terms does not appear in the statement of the corollary. However, we need to keep the original estimate of the Lipschitz constant for the corrector $D\psi_1^\delta$, namely ${C_\rho^{x_1, u_1}} $.
\end{proof}

\section{The homogenization}\label{sec:homog}

We establish in this section the homogenization result for problem \eqref{eq:HJ-eps}. More precisely, we show that the viscosity solutions $\big(u^\veps\big)_{\veps>0}$ of \eqref{eq:HJ-eps} converge locally uniformly to the solution of the averaged equation \eqref{eq:HJ-eff}.
The proof uses the perturbed test function method,  which is standard and we do not detail here. Nonetheless, the uniqueness of the limit for convergent subsequences is not straightforward, since linearization does not go hand in hand with the viscosity solution theory approach and difficulties imposed by the $x$ dependence and the behaviour of the measure near the singularity might appear. This is due to the fact that the effective Hamiltonian is implicitly defined and its linearization is based on the variable-dependence given in Corollary \ref{cor:eff-H}. The Lipschitz regularity result and in particular the fine estimate of the Lipschitz constant play a central role in the linearization procedure. 

We start by noting that, in view of Corollary \ref{cor:eff-H}, the regularity results for weakly elliptic nonlocal operators obtained in \cite{Bci11, Bcci12} apply and solutions for the effective problem are H\"older continuous.

\begin{prop}\label{prop:reg-eff}
Let  $(f^a)_{a\in\mathcal A}$ and $(b^a)_{a\in\mathcal A}$ be  two families of bounded functions on $\R^{2d}$, satisfying  $(H1)$ with respect to both variables with $\alpha,\beta\in(0,1]$. Let $(K^a)_{a\in\mathcal A}$ be a family of kernels satisfying $(K1)-(K3)$ with  $\gamma\in(0,1]$. 
Then any  bounded continuous viscosity solution $u:\R^d\times[0,T]\to\R$ of \eqref{eq:HJ-eff} is H\"older continuous in space, i.e. there exists $\tau\in(0, 1)$ such that for all $t\in[0,T]$, $u(\cdot, t)\in\mathcal C^{0,\tau}(\R^d)$. 
\end{prop}

\begin{proof}
Note that we cannot literally apply Theorem 1 of \cite{Bci11} as we do not have an explicit formulation for $\overline H$ and hence  the ellipticity-growth condition  $(H)$ of \cite{Bci11} cannot be checked. Nonetheless, it is enough to remark that the right hand side of  the ellipticity-growth condition $(H)$ plays the central role in getting the regularity. Making use of Corollary \ref{cor:eff-H} we get a similar expression for the effective Hamiltonian $\overline H$. Namely, in our case the functions  $\Lambda_1\equiv 0$ and $\Lambda_2\equiv 1$ and the nonlocal difference $l_1-l_2$ in $(H)$ is  just the explicit expression $\sup_{\substack{a\in\mathcal A\\ \xi\in\Pi^d}} \big(\mathcal L^a(x_1,\xi,u_1) - \mathcal L^a(x_2,\xi,u_2) \big)$ (which could have also been directly written in  \cite{Bci11}). The only term we need to exploit in our case is the (first) one having a nonlinear dependence between the space variable $x$, the gradient variable $p$ and the function $u$ - given in terms of the constant $C_\rho^ {x,u}$. 

Recall that in order to prove H\"older regularity a radial penalty function of the form $\varphi(|x-y|) = L |x-y|^ \tau$ is considered and estimates are made within the viscosity inequalities. In our case, it is enough to consider the following parameters in Corollary \ref{cor:eff-H} above, $\rho = \rho_0 |x_1-x_2|$, and $u_1$ a fonction satisfying $p_1 = Du_1(x_1)$, and $|| D^2u_1 ||_{L^\infty(B_\rho(x_1))} \leq C |p_1| \rho^{-1}$. The constant $C_\rho^{x_1, u_1} $ then becomes
$ C_\rho^{x_1, u_1}  = C |p_1| + |p_1| \; \big| \ln\big(\rho_0 |x_1-y_1|\big) \big| + C \rho_0^{-1} |x_1-y_1|$
and the first term in the bound of $\overline H(x_2,p_2,u_2) -  \overline H(x_1,p_1,u_1) $ is given, up to a modification fo the universal constant $C$, by 
\[ C \Big( \left (1 + |p_1| + |p_1| \; \big| \ln\big(\rho_0 |x_1-y_1|\big) \big| + C \rho_0^{-1} |x_1-y_1| \right)^{\frac{1}{1+\sigma}} + \big(1+|p_1| \big) \Big) |x_1-x_2|^{{\min(\alpha,\beta)}}. \]	
This is enough to reach the same conclusion.
\end{proof}

A priori regularity of the solution further permits to establish a linearization result for the effective problem, which is formulated in terms of the extremal Pucci operators
\[\mathcal M^+(x,\phi) = \sup_{a\in\mathcal A}\sup_{\xi\in\R^d} \mathcal L^a(x,\xi, \phi), \;\;\;
  \mathcal M^-(x,\phi) = \sup_{a\in\mathcal A}\inf_{\xi\in\R^d} \mathcal L^a(x,\xi, \phi). \]

\begin{prop}\label{prop:linearization}
Let  $(f^a)_{a\in\mathcal A}$ and $(b^a)_{a\in\mathcal A}$ be  two families of bounded functions on $\R^{2d}$, satisfying $(H0)$ and $(H1)$ with respect to both variables with $\tilde{d}=2d$ and $\alpha=\beta=1$. Let $(K^a)_{a\in\mathcal A}$ be a family of kernels satisfying $(K0)-(K3)$ with $\gamma=1$. Assume in addition that $(Ks)$ or $(Kns)$  hold. 
Let $u\in USC(\R^d\times[0,T])$  and $ v\in LSC(\R^d\times[0,T])$  be respectively a viscosity subsolution and viscosity supersolution of  equation \eqref{eq:HJ-eff}.
\begin{enumerate} 
\item If  $v(\cdot,t) \in \mathcal C^{0,\tau}(\R^d)$ for all $t\in[0,T]$ with $\tau\in(0,1)$,
	then $w = u-v$ is a viscosity subsolution of 
	\begin{equation}\nonumber			
	w_t - \mathcal M^+(x,w(\cdot,t)) - B | Dw | = 0 \quad \text{ in } \quad \R^d\times[0,T],
	\end{equation}	
\item If  $u(\cdot,t) \in \mathcal C^{0,\tau}(\R^d)$ for all $t\in[0,T]$ with  $\tau\in(0,1)$,
	then $w = v-u$ is a viscosity supersolution of 
	\begin{equation}\nonumber
	w_t + \mathcal M^-(x,w(\cdot,t)) + B | Dw | = 0 \quad \text{ in } \quad \R^d\times[0,T],
	\end{equation}	
\end{enumerate}
where $B = \sup_{a\in\mathcal A} || b^a ||_\infty. $
\end{prop}

\begin{proof}
Fix $(x_0,t_0)\in\R^d\times(0,T)$ and $\rho'>0$ and let $\varphi\in\mathcal C^{2}(\R^d\times [0,T])$ such that $w-\varphi$ has a strict maximum at $(x_0,t_0)$ in $B_{\rho'}(x_0,t_0)$. We want to show that
\begin{equation}\label{eq:HJ-eff-lin-visc}
	\varphi_t(x_0,t_0) - \mathcal M^+\big(x_0,\1_{B_{\rho'}(x_0)}\varphi(\cdot, t_0) + 
	 \1_{B_{\rho'}^c(x_0)} w(\cdot, t_0)\big)  - B | D\varphi(x_0,t_0) | \leq 0.
\end{equation}
Consider, for $\eps>0$, the function
\[\phi(x,y,t,s) = \varphi(x,t) + \frac{|x - y|^2}{\eps^2} + \frac{(t-s)^2}{\eps^2} + \psi_\zeta(x),\]
where $\psi_\zeta(x) := \psi(\zeta x)$ is a localisation term, with a choice of  a smooth function $\psi \geq 0$, satisfying $\psi = 0$ in $B$ and $\psi \geq 1 + \| u \|_\infty + \| v \|_\infty + \| \varphi \|_\infty$ outside $B_2$. 

Since $(x_0,t_0)$ is a strict global maximum for $u(x,t)-v(x,t)-\varphi(x,t)$, for $\eps$ sufficiently small, there exists a sequence of points  $(x_\eps, y_\eps, t_\eps, s_\eps)$ which are local maxima respectively for 
\[ \Phi(x,y,t,s) := u(x,t) - v(y,s) - \phi(x,y,t,s).\] 
It follows, from the inequality $\Phi(x_\eps,x_\eps, t_\eps, s_\eps)\leq \Phi(x_\eps,y_\eps, t_\eps, s_\eps)$ and the regularity of $v$, that
\[ \frac{|t_\eps - s_\eps|^2}{\eps^2} 
   \leq v(x_\eps, s_\eps) - v(y_\eps, s_\eps) 
   \leq  2 ||v||_\infty, \]
   and 
 \[ \frac{|x_\eps - y_\eps|^2}{\eps^2} 
   \leq v(x_\eps, s_\eps) - v(y_\eps, s_\eps) 
   \leq C |x_\eps - y_\eps|^{{\tau}}. \]
Therefore, the following holds
\begin{equation}\label{eq:HJ-eff-point-est} 
	|t_\eps - s_\eps| \leq C\eps^{2},\quad \quad \quad 
	|x_\eps - y_\eps| \leq C\eps^{2/(2-\tau)}. 
\end{equation}	
In particular, $(x_\eps, y_\eps, t_\eps, s_\eps)\to(x_0,x_0,t_0,t_0)$ as $\eps\to 0$ for any fixed $\zeta>0$. To simplify notation we dropped their dependence in $\zeta$.

Let 
\[ \phi^{u}(x,t) = v(y_\eps,s_\eps) + \phi(x,y_\eps, t, s_\eps), \]
\[ \phi^{v}(y,s) = u(x_\eps,t_\eps) - \phi(x_\eps, y, t_\eps, s),\]
where for convenience of notations we have dropped the $\eps$-dependence in $\phi^u$ and $\phi^v$. Note that $(x_\eps, t_\eps)$ is a maximum of $u - \phi^u$ in $B_\rho(x_\eps, t_\eps)$, whereas $(y_\eps, s_\eps)$ is a minimum of $v - \phi^v$ in  $B_\rho(y_\eps,s_\eps)$, for $\rho\in(0,\rho')$ sufficiently small. We will eventually choose $\rho = \eps^r$ with $r>0$ yet to be determined, and let $\eps\to0$, then $\zeta\to0$.

Let
\[ {\tilde u_{\rho}}(\cdot,t) = \1_{B_\rho(x_\eps)}\phi^u(\cdot, t)  + \1_{B_\rho^c(x_\eps)} u(\cdot, t), \]
\[ {\tilde v_{\rho}}(\cdot,s) = \1_{B_\rho(y_\eps)}\phi^v(\cdot, s) + \1_{B_\rho^c(y_\eps)} v(\cdot, s) .\]
The viscosity inequalities for the sub and supersolution then read
\begin{eqnarray*}
	\phi^u_t(x_\eps,t_\eps) + 
	\overline H \big(x_\eps, D\phi^u(x_\eps,t_\eps), {\tilde u_{\rho}}(\cdot, t_\eps) \big) \leq 0,\\
	\phi^v_t(y_\eps,s_\eps) + 
   	\overline H\big(y_\eps, D\phi^y(y_\eps,s_\eps),  {\tilde v_{\rho}}(\cdot, s_\eps)\big)\geq 0.
\end{eqnarray*}
Subtracting the two inequalities above, it follows, in view of Corollary \ref{cor:eff-H}, that
\begin{eqnarray}\label{eq:HJ-eff-est-1}
	\varphi_t(x_\eps,t_\eps) &\leq & 	
	\overline H\big(y_\eps, D\phi^v(y_\eps,s_\eps), {\tilde v_{\rho}}(\cdot, s_\eps) \big) -
	\overline H \big(x_\eps, D\phi^u(x_\eps,t_\eps),  {\tilde u_{\rho}}(\cdot, t_\eps) \big) \\
	\nonumber &\leq & 	
	\mathcal {Q}_{\eps}^u |x_\eps - y_\eps| + B \left| D\varphi(x_\eps,t_\eps)\right| +
	 \sup_{\substack{a\in\mathcal A\\ \xi\in\R^d}} 
	 \Big\{ \mathcal L^a(x_\eps,\xi,{\tilde u_{\rho}}(\cdot,t_\eps)) -  
	\mathcal L^a(y_\eps,\xi,{\tilde v_{\rho}}(\cdot,s_\eps)) \Big\},
\end{eqnarray}
where $B=\sup_{a \in \mathcal{A}}||b^a||_\infty$ and 
\[ \mathcal {Q}_\eps^u : = C \big(1+ | D{\phi^u}(x_\eps,t_\eps)| + 
   C_{\rho,\eps}\big)^{\frac{1}{1+\sigma}} +  1 + | D{\phi^u}(x_\eps,t_\eps)|,\]
with $C>0$ a universal constant and  $C_{\rho,\eps}$ a constant depending on $\tilde u_\rho$ given by \eqref{eq:cell-Lip-corr-const}.

Each of the terms above is further estimated as $\eps\to 0$.  We start with the first term.  Note that the constant $C_{\rho,\eps}$ herein translates into
\begin{eqnarray*}
	C_{\rho,\eps} & = &
	||D^2 \phi^u||_{L^\infty(B_\rho(x_\eps,t_\eps))}\; \rho +
	|D\phi^u(x_\eps, t_\eps)|\; |\ln(\rho)| + 
   	|| u ||_{\infty} \rho^{-1}\\
	& \leq &
	\tilde C\Big(\big(1 + \eps^{-2} + o_\zeta(1)\big) \rho + 
			  \big(1 + |p_\eps| + o_\zeta(1)\big)\; |\ln(\rho)| + 
			  \rho^{-1}\Big),
\end{eqnarray*}
where $p_\eps = (x_\eps- y_\eps)/\eps^2$ and $\tilde C>0$ is a constant depending on $|| \varphi ||_{C^{2}(B_{\rho'}(x_0,t_0))}$ and $||u||_\infty$. Using \eqref{eq:HJ-eff-point-est} and the fact that  we will chose $\rho$ of the form $\rho = \eps^r$ with $r>0$ such that all the terms will be bounded, it follows, up to a modification of the  constant $C$, that
\begin{eqnarray*}
	\mathcal {Q}_{\eps}^u  |x_\eps - y_\eps| 
	& \le &C
	\Big(1+ |p_\eps| + \big(1 + \eps^{-2}\big) \rho + 
				    \big(1 + |p_\eps| + o_\zeta(1)\big)\; |\ln(\rho)| + 
				    \rho^{-1} + o_\zeta(1)\Big)^{\frac{1}{1+\sigma}}  |x_\eps - y_\eps|\\
	&&			    + C \left(1 + |p_\eps| \right) |x_\eps - y_\eps| + o_\zeta(1)\\
	& \le &C 
	\Big( o_\eps(1) + |x_\eps - y_\eps|^{\sigma+1} \eps^{-2} \rho + 
				   |x_\eps - y_\eps|^{\sigma+2} \eps^{-2} |\ln(\rho)| + 
				   |x_\eps - y_\eps|^{\sigma+1}  \rho^{-1}\Big)^{\frac{1}{1+\sigma}}+ \\
	 &&o_\eps(1) + o_\zeta(1) \\
	& \le &C 
	\Big( o_\eps(1) + \eps^{\frac{2(\sigma+1)}{2-\tau}-2+r}  + 
				  \eps^{\frac{2(\sigma+2)}{2-\tau}-2} |\ln(\eps)| + 
				  \eps^{\frac{2(\sigma+1)}{2-\tau}-r}\Big)^{\frac{1}{1+\sigma}}+ 
	 o_\eps(1) + o_\zeta(1).
\end{eqnarray*}
Let $r = \frac{2(\sigma+1)}{2-\tau} -  \frac{\tau}{2-\tau}$ and choose $\sigma>1-\tau/2$. Note that we strongly rely on the estimate of the Lipschitz constant for the corrector to control the terms above : the exponent $\sigma$ in $\mathcal Q_\eps^u$ can be chosen arbitrarily close to one. The above estimate then writes
\begin{eqnarray}\label{eq:HJ-eff-est-2}
	\nonumber \mathcal {Q}_{\eps}^u  |x_\eps - y_\eps|  & \le &C 
	\Big( o_\eps(1) + \eps^{\frac{4\sigma+\tau}{2-\tau}}  + 
				  \eps^{\frac{2(\sigma+\tau)}{2-\tau}} |\ln(\eps)| + 
				   \eps^{\frac{\tau}{2-\tau}}\Big)^{\frac{1}{1+\sigma}}+ 
	 o_\eps(1) + o_\zeta(1)\\
	 & = &  o_\eps(1) + o_\zeta(1).
\end{eqnarray}

We now estimate the nonlocal difference. To this end, we split the domain of integration into $B_\rho$, $B_{\rho'}\setminus B_{\rho}$ and $B_{\rho'}^c$ and evaluate ${\mathcal T^a}(x_\eps, y_\eps) : =\mathcal L^a(x_\eps,\xi,{\tilde u_{\rho}}) -  \mathcal L^a(y_\eps,\xi,{\tilde v_{\rho}}) $. As usual, we use the notation ${\mathcal T^a}[D]$ to specify the domain of integration $D$ on which the nonlocal difference is computed.
\begin{eqnarray*}
	{\mathcal T^a}[B_\rho](x_\eps, y_\eps)
	&=& \int_{B_\rho} \big( \phi^u(x_\eps + z,t_\eps)  -  \phi^u(x_\eps,t_\eps) -   	
		D\phi^u(x_\eps,t_\eps)\cdot z \big)K(\xi,z)dz- \\
	&& 	\int_{B_\rho}  \big( \phi^v(y_\eps + z,s_\eps) -   \phi^v(y_\eps,s_\eps) -   	
		D\phi^v(y_\eps,s_\eps)\cdot z \big) K(\xi,z)dz  \\	
	&=& 	\int_{B_\rho}  \big( \varphi(x_\eps + z,t_\eps)  -  \varphi(x_\eps,t_\eps) -   	
		D\varphi(x_\eps,t_\eps)\cdot z \big)K(\xi,z)dz +\\
	&&	 \frac{2}{\eps^2}\int_{B_\rho}|z|^2 K(\xi,z)dz+ 
		 \int_{B_\rho}  \big( \psi_\zeta(x_\eps + z)  -  \psi_\zeta(x_\eps) +   	
		D\psi_\zeta(x_\eps)\cdot z \big)K(\xi,z)dz\\	
	&\le&\mathcal L^a[B_\rho](x_\eps,\xi,\varphi(\cdot,t_\eps))+  
		\frac{2}{\eps^2}\int_{B_\rho}|z|^2 K(\xi,z)dz+
		\mathcal L^a[B_\rho](x_\eps,\xi,\psi_\zeta).
\end{eqnarray*}
To estimate the nonlocal difference on $B_{\rho'}\setminus B_\rho$ and on $B_{\rho'}^c$ we use again the maximum property and deduce from the inequality $\Phi(x_\eps+z,y_\eps+z,t_\eps,s_\eps) \leq \Phi(x_\eps,y_\eps,t_\eps,s_\eps)$ that

\begin{eqnarray*}
{\mathcal T^a}[B_{\rho'}\setminus B_{\rho}](x_\eps, y_\eps)
	&=& 	\int_{B_{\rho'}\setminus B_{\rho}} \big( u(x_\eps + z,t_\eps)  -  u(x_\eps,t_\eps) -   	
		\1_B (z) D\phi^u(x_\eps,t_\eps)\cdot z\big)K(\xi,z)dz - \\
	&&	\int_{B_{\rho'}\setminus B_{\rho}}  \big( v(y_\eps + z,s_\eps)  -  v(y_\eps,s_\eps) 	
		-\1_B (z) D\phi^v(y_\eps,s_\eps)\cdot z \big) K(\xi,z)dz  \\	
	&\le& \int_{B_{\rho'}\setminus B_{\rho} }
		\Big(	\big(\varphi(x_\eps + z,t_\eps)  -  \varphi(x_\eps,t_\eps)  \big) +
			\big( \psi_\zeta(x_\eps + z)  - \psi_\zeta(x_\eps) \big)  \\
	&&	\hspace{2.3cm}
		- \1_B (z) \big(D\varphi (x_\eps,t_\eps) + D\psi_\zeta(x_\eps) \big)\cdot z \Big) K(\xi,z)dz\\
	&=& 	\mathcal L^a \big[B_{\rho'}\setminus B_{\rho}\big](x_\eps,\xi,\varphi(\cdot,t_\eps))+ 
		\mathcal L^a \big[B_{\rho'}\setminus B_{\rho}\big](x_\eps,\xi,\psi_\zeta),		
\end{eqnarray*}
whereas
\begin{eqnarray*}
	{\mathcal T^a}[B_{\rho'}^c](x_\eps, y_\eps)
	=\int_{B_{\rho'}^c}  \Big(	 \big( u(x_\eps + z,t_\eps)  -  u(x_\eps,t_\eps)\big) 
						-\big( v(y_\eps + z,s_\eps)  -  v(y_\eps,s_\eps)\big) - \\
	- \1_B (z) \big(D\varphi (x_\eps,t_\eps) + D\psi_\zeta(x_\eps) \big)\cdot z \Big) K(\xi,z)dz.
\end{eqnarray*}
The overall estimate becomes
\begin{eqnarray*}
{\mathcal T^a}(x_\eps, y_\eps)
	&\le &\mathcal L^a[B_{\rho'}](x_\eps,\xi,\varphi(\cdot,t_\eps))+ 
	\frac{2}{\eps^2}\int_{B_\rho}|z|^2 K(\xi,z)dz + o_\zeta(1)+ \\
	&&\int_{B_{\rho'}^c} \Big(\big( u(x_\eps + z,t_\eps)  -  u(x_\eps,t_\eps)\big) 
					-\big( v(y_\eps + z,s_\eps)  -  v(y_\eps,s_\eps)\big) - \\
	&&\hspace{4.5cm}- \1_B (z) \big(D\varphi (x_\eps,t_\eps) \big)\cdot z \Big) K(\xi,z)dz.
\end{eqnarray*}
Let $\rho = \eps^{\frac{2(\sigma+1) - \tau}{2-\tau}}$ and $\sigma>1-\tau/2$ as above. In view of  $(Ks)$ or $(Kns)$, it follows that
\[ \frac{2}{\eps^2}\int_{B_\rho}|z|^2 K(\xi,z)dz\leq 
  C_K\frac{2}{\eps^2} \rho 
  \leq \tilde C_K  \eps^{\frac{2\sigma}{(2-\tau)}-1}=o_\eps(1). \]
Employing the dominated convergence theorem and the semi-continuity of $u$ and continuity of $v$, it follows that as $\eps\to 0$,
\begin{eqnarray*}
	\limsup_{\eps \to 0}\mathcal T^a(x_\eps, y_\eps)
	&\le& \mathcal L^a[B_{\rho'}](x_0,\xi,\varphi(\cdot,t_0))+ \\
	&&\int_{B_{\rho'}^c} \Big( w(x_0 + z,t_0)  -  w(x_0,t_0) - \1_B (z) D\varphi (x_0,t_0)\cdot z \Big) K(\xi,z)dz
		+ o_\zeta(1)\\
	& = &\mathcal L^a \big[B_{\rho'}   \big](x_0,\xi,\varphi(\cdot,t_0))+
		\mathcal L^a \big[B_{\rho'}^c\big](x_0,\xi,        w(\cdot,t_0))+ o_\zeta(1).
\end{eqnarray*}
Therefore, the following overall estimate holds for the nonlocal difference
\begin{eqnarray}
	\label{eq:HJ-eff-est-3}
	\limsup_{\eps\to 0} \sup_{\substack{a\in\mathcal A,\\ \xi\in\R^d}} 
	\Big\{ \mathcal L^a(x_\eps,\xi,{\tilde u_{\rho}}(\cdot,t_\eps)) -  
		 \mathcal L^a(y_\eps,\xi,{\tilde v_{\rho}}(\cdot,v_\eps)) \Big\}\\
	\nonumber
	\le \mathcal M^+\big(x_0,\varphi(\cdot, t_0)\1_{B_{\rho'}(x_0)} + w(\cdot, t_0) \1_{B_{\rho'}^c(x_0)}\big)
	+ o_\zeta(1).
\end{eqnarray}
We conclude, from equations \eqref{eq:HJ-eff-est-1}-\eqref{eq:HJ-eff-est-3}, letting $\eps\to0$ and then $\zeta \to 0$ that  \eqref{eq:HJ-eff-lin-visc} holds.
\end{proof}

\begin{remark}
Lipschitz regularity of the data is necessary to linearize. This appears already in \cite{C12} when a strong comparison between subsolutions and supersolutions is shown, for L\'evy-It\^o integro-differential equations. However, for more general L\'evy measures, as above, the result is unknown. We are able here to prove it for the effective Hamiltonian since there is an explicit dependence on the Lipschitz bound of the corrector. Recalling that the Lipschitz estimate  \eqref{eq:Lip} depends only on the exponent $\alpha$ of the source term (and other constants), which in the case of the corrector is $\tilde f^ a$ and hence it involves all the datum, it is crucial that $\sigma \in(0,\min(\alpha,\beta,\gamma))$ is as close as possible to $1$, from where the requirement that $\alpha, \beta,\gamma$ ought to be  $1$.
\end{remark}

We are now in shape of proving the main homogenization result.

\begin{theorem}\label{thm:conv}
Let  $(f^a)_{a\in\mathcal A}$ and $(b^a)_{a\in\mathcal A}$ be  two families of bounded functions on $\R^{2d}$, satisfying $(H0)$ and $(H1)$ with respect to both variables with $\tilde{d}=2d$ and with $\alpha=\beta=1$. Let $(K^a)_{a\in\mathcal A}$ be a family of kernels satisfying $(K0)-(K3)$ with $\gamma=1$. Assume in addition that $(Ks)$ or $(Kns)$  hold.
Then, the viscosity solutions $\big(u^\veps\big)_{\veps>0}$ of \eqref{eq:HJ-eps} converge locally uniformly to the unique, bounded continuous  viscosity solution $u$ of \eqref{eq:HJ-eff}.
\end{theorem}

\begin{proof} 
Note first that, by means of  a vanishing coercivity argument (as in the proof of Lemma \ref{lem:cell-approx}), for each $\veps >0$ problem \eqref{eq:HJ-eps} admits a bounded continuous viscosity solution $u^\veps$, which  in view of Theorem \ref{thm:Lip-ev}, is Lipschitz continuous for all times $t\in(0,T)$. Comparison principle given in Proposition \ref{prop:comp-ev} for the class of Lipschitz functions (in space) further asserts the uniqueness of $u^\veps$. 

It is easy to see that the sequence is uniformly bounded. If $M = \sup_{a\in\mathcal A }  || f^a||_\infty$, note that 
\[ \overline u(x,t) =   || u_0  ||_\infty +  M t, \quad \underline u(x,t) = - || u_0  ||_\infty -  M t\]
are respectively supersolutions and subsolutions of $\eqref{eq:HJ-eps}$. Hence
\[ \sup_{\veps>0}|| u^\veps ||_{\infty} \le || u_0  ||_\infty +  M T.\]

However, since the nonlocal operator is only weakly elliptic, in the sense of assumption $(K3)$, uniform  H\"older or Lipschitz estimates are not available.  In order to show that the sequence  $\big(u^\veps\big)_{\veps>0}$ converges uniformly to a viscosity solution of the effective problem, we employ half-relaxed limits, introduced by Ishii \cite{I89} and Barles and Perthame \cite{Bp87, Bp88}. Let
\begin{eqnarray*}
	u^*(x,t)  = \limsup_{\substack{\veps\to 0\\(y,s)\to(x,t)}} u^{\veps}(y,s) \;\;\;
	u_*(x,t)  = \liminf_{\substack{\veps\to 0\\(y,s)\to(x,t)}} u^{\veps}(y,s).
\end{eqnarray*}
Then $u^*$ is bounded and upper semi-continuous, $u_*$ is bounded and lower-continuous. By definition, for all $(x,t)\in \R^d\times(0,T)$,
\[ u_*(x,t) \leq u^*(x,t).\]

Moreover, in view of the comparison principle for equation \eqref{eq:HJ-eps} and the fact that $u^\veps$ are Lipschitz continuous in space, it follows  that there exists a modulus of continuity independent of $\veps$, given by $\omega(t) = Mt $ for $t\in[0,T]$, with $M$ as above, such that
\[\sup_{x\in\R^d} |u^\veps(x,t) - u_0(x)| \leq \omega(t). \]
Hence, the following holds for the initial condition
\[ u^*(x,0) \leq u_0(x) \leq u^*(x,0).\]

Employing the perturbed test function method, it is possible to check that $u^*$ is viscosity subsolution and  $u_*$ is viscosity supersolution of the effective problem \eqref{eq:HJ-eff}, in the sense of Definition~\ref{def:H-sol}. Nonetheless, the lack of comparison principle for the effective problem does not allow us to conclude directly that  $u^*\geq u_*$.

To overcome this difficulty, we note that, in view of the structural properties of the effective Hamiltonian, the same vanishing coercivity argument as in the proof of Lemma \ref{lem:cell-approx} applies  and we conclude, in view of Proposition \ref{prop:reg-eff}, the existence of a viscosity solution $u$ of the effective problem \eqref{eq:HJ-eff}, which is $\tau-$H\"older continuous in space. The linearization result stated in Proposition \ref{prop:linearization}, applied on one hand to $u^*$ and $u$, and on the other hand to $u$ and $u_*$, further gives
\[u^*(x,t)\leq u(x,t) \leq u_*(x,t), \text{ for all } x\in\R^d, t\in [0,T].\] 
Hence $u^*=u_*=u$ and the whole sequence converges locally uniformly in $\R^d\times [0,T]$. 
\end{proof}

\begin{remark}
In the uniformly elliptic case, i.e. when the kernel satisfies
\begin{equation*}
	\frac{1}{C_K |z|^{d+1}} \leq K^a(\xi,z) \leq \frac{C_K}{|z|^{d+1}}  \quad \mbox{ for } \xi\in\R^d, \ z \in B \setminus \{  0 \}, 
\end{equation*}
the uniform convergence is immediate, in view of the space-time regularity results of Chang-Lara and Davila (see Corollary 7.1 in~\cite{CLd16}). More precisely, in the proof above the equi-bounded family of solutions $\big(u^\veps\big)_{\veps>0}$  becomes uniformly $\tau$-H\"older continuous in space and time, i.e. there exists $\tau\in(0,1)$ such that
\[ \sup_{\veps>0}|| u^\veps ||_{\mathcal C^{0,\tau}(\R^d\times[0,T])} <\infty.\]
 In view of Arzela-Ascoli theorem, there exists a subsequence $(\veps_k)_{k>0}$, such that $\left(u_{\veps_k}\right)_{k}$ converges locally uniformly in $\R^d\times[0,T]$ to a function $u\in\mathcal C^{0,\tau}(\R^d\times[0,T])\cap L^\infty(\R^d\times[0,T])$.
\end{remark}

\section{Appendix}
General comparison results have been established by Barles and Imbert in \cite{Bi08} for L\'evy-It\^o integro-differential operators, but it continues to be an open problem for  nonlocal L\'evy operators with $x$ dependent kernels. However, under {\em a priori} Lipschitz regularity assumption on the sub/super-solutions, comparison can be established by standard arguments. Though results apply for parabolic problems as well, we give a sketch of the proof in the stationary case, to simplify ideas.

\begin{prop}\label{prop:comp}
Let  $(f^a)_{a\in\mathcal A}$ and $(b^a)_{a\in\mathcal A}$ be  two families of bounded functions on $\R^{d}$ satisfying $(H1)$  with $\alpha,\beta \in (0,1]$. Let $(K^a)_{a\in\mathcal A}$ be a family of kernels satisfying $(K1)$ and $(K3)$ with $\gamma\in(\frac12,1]$. If $u\in USC(\R^d)$ and $v\in LSC(\R^d)$  are respectively  a bounded viscosity subsolution and a bounded viscosity supersolution of equation \eqref{eq:HJ-d}, such that $u\in C^{0,1}(\R^d)$ or $v \in \mathcal C^{0,1}(\R^d)$, then $ u\leq v$ on $\R^d$.
\end{prop}

\begin{proof}
We argue by contradiction and assume that $M:=\sup_{x\in\R^d}(u(x)-v(x)) >0$. Doubling the variables we consider
\[M_{\eps,\zeta} = \sup\{ u(x) - v(y) - \phi_\eps(x-y) - \psi_\zeta(x)\}, \]
where $\eps,\zeta$ are small parameters that will eventually go to $0$. The penalization function 
$\phi_\eps: \R^d \to \R_+$ is given by  
$$
\phi_\eps(x-y):= \varphi\left(\frac{|x-y|^2}{\epsilon^2} \right),
$$
where $\varphi$ is a smooth  nonnegative function on $\R_+$, with $ ||\varphi ||_\infty$, $ ||\varphi' ||_\infty $ and $ ||\varphi'' ||_\infty $ all finite and
\[ \varphi(s) = \left\{ \begin{array}{ll} 
	s			& \text{ if } s \leq s_0  \\
	2 || u ||_\infty +1 	& \text{ if } s \geq 2 s_0.
\end{array} \right.\]
The localization function $\psi_\zeta$ is given by $\psi_\zeta(x) = \psi(\zeta x)$, with
$\psi \in \mathcal C^2(\R^d;\R_+)$ with $ ||\psi ||_\infty $, $ || D\psi ||_\infty $ and $ || D^2\psi ||_\infty $ all finite, such that 
\[ \psi(x) = \left\{ \begin{array}{ll} 
	0						& |x| \leq 1 \\
	|| u ||_\infty + || v ||_\infty + 1	& |x| \geq 2.
\end{array} \right.\]
Since the localization function only gives terms in $o_\zeta(1)$, we drop the dependence in $\zeta$ in what follows. In view of the properties of the localisation term $\psi_\zeta$, the supremum $M_{\eps, \zeta}$ is actually a maximum, achieved at a point that we denote $(x_\eps, y_\eps)\in B_{s_0}\times B_{s_0}$. For $\eps$ small enough
\begin{equation}\nonumber
	\frac{M}{2}\leq M_{\eps,\zeta} \leq u(x_\eps)-v(y_\eps) \le || u ||_\infty + || v ||_\infty ,
\end{equation}	
whereas the maximum property together with the assumption that  $u\in C^{0,1}(\R^d)$ give
\begin{equation*}
\frac{|x_\eps-y_\eps|^2}{\epsilon^2}\leq |u(x_\eps)-u(y_\eps)|\leq C |x_\eps -y_\eps|,
\end{equation*}
hence $|x_\eps-y_\eps| \leq C \eps^2$. Let $a_\eps = x_\eps-y_\eps$, 
$\displaystyle p = \frac{x_\eps-y_\eps}{\eps^2}$, $q = D\psi_\zeta(y_\eps)$.  
Denote 
\[ \phi^u(x):= v(y_\eps) + \phi_\eps(x-y_\eps)+\psi_\zeta(x), \]
\[ \phi^v(y):= u(x_\eps) -  \phi_\eps(x_\eps-y)-\psi_\zeta(x_\eps),\]
and observe that
\[ D\phi_\eps(a_\eps) = \varphi'\left(\frac{|a_\eps|^2}{\eps^2}\right) p, \;\;\;
   D\phi^u(x_\eps):=  D\phi_\eps(a_\eps) + q, \;
   D\phi^v(y_\eps):=  D\phi_\eps(a_\eps).\]
In view of the maximum property, there exists  $\rho\in(0,\min(1,s_0))$ sufficiently small such that $x_\eps$ is a local maximum for $u-\phi^u$ in $B_\rho(x_\eps)$, and  $y_\eps$ is a local minimum for $v- \phi^v$ in $B_\rho(y_\eps)$. It follows from the viscosity inequalities that, for any $\nu>0$, there exists $a \in \mathcal{A}$ such that, for all $0<\rho'<\rho$, 
\begin{eqnarray*}
	\delta u(x_\eps)-\mathcal{I}^a[B_{\rho'}](x_\eps,  \phi^u(x_\eps))-\mathcal{I}^a[B_{\rho'}^c](x_\eps, u)
	-b^a(x_\eps) \cdot D\phi^u(x_\eps)-f^a(x_\eps)&\leq& 0,\\
	\delta v(y_\eps)-\mathcal{I}^a[B_{\rho'}](y_\eps,  \phi^v(y_\eps))-\mathcal{I}^a[B_{\rho'}^c](y_\eps,  v)
	-b^a(y_\eps) \cdot D\phi^v(y_\eps)-f^a(y_\eps)\nonumber &\geq& -\nu.\\
\end{eqnarray*}
Denote 
\begin{eqnarray*}
	\mathcal{T}^a[B_{\rho'}]	(x_\eps,y_\eps, \phi) & := & 
		\mathcal{I}^a[B_{\rho'}](x_\eps,  \phi^u(x_\eps))-\mathcal{I}^a[B_{\rho'}](y_\eps, \phi^v(y_\eps)),\\
	\mathcal{T}^a[B_{\rho'}^c](x_\eps,y_\eps, u, v) & :=& 
	\mathcal{I}^a[B_{\rho'}^c](x_\eps, u)-\mathcal{I}^a[B_{\rho'}^c](y_\eps,  v).
\end{eqnarray*}
Subtracting the two inequalities, it follows that
\begin{eqnarray}\label{eq:comp-visc}
	\delta u(x_\eps)-\delta v(y_\eps) - \nu& \leq & 
	\mathcal{T}^a[B_{\rho'}](x_\eps,y_\eps, \phi)+\mathcal{T}^a[B_{\rho'}^c](x_\eps,y_\eps, u, v) + \\ \nonumber&& 
	\varphi'\left(\frac{|a_\eps|^2}{\eps^2}\right) \big( b^a(x_\eps)-b^a(y_\eps) \big)\cdot  p  +
	b^a(x_\eps)\cdot q + \left( f^a(x_\eps)-f^a(y_\eps)\right) ,
\end{eqnarray}
whose last terms are further bounded by, in view of assumption $(H1)$ and previous notations,
\[ C_b \varphi'\left(\frac{|a_\eps|^2}{\eps^2}\right) |a_\eps|^\beta |p| +
	||b^a||_\infty |q| + C_f |a_\eps|^\alpha  \leq
	C_b ||\varphi' ||_\infty \eps^{2\beta}  + C_f \eps^{2\alpha} + o_\zeta(1) = o_\eps(1) + o_\zeta(1).\]

In order to estimate the nonlocal terms, we make use of the following inequalities coming from the maximum property
\begin{eqnarray*}
	u(x_\eps+z) - u(x_\eps) - \left(D\phi_\eps(a_\eps)+q\right)\cdot z &\leq& 
	\phi_\eps(a_\eps+z) - \phi_\eps(a_\eps) - D\phi_\eps(a_\eps)\cdot z + \\ &&
	\psi_\zeta(y_\eps +z) + \psi_\zeta(y_\eps) - q \cdot z\\
	-\left( v(y_\eps +z) - v(y_\eps) - D\phi_\eps(a_\eps)\cdot z\right) &\leq& 
	\phi(a_\eps -z) - \phi(a_\eps) + D\phi_\eps(a_\eps)\cdot z.
\end{eqnarray*}
Letting first $\rho' \to 0$, it is immediate to see that the term $\mathcal{T}^a[B_{\rho'}](x_\eps,y_\eps, \phi)$ is $o_{\rho'}(1)$. To simplify notations hereafter, we write $\mathcal T^a{(x_\eps, y_\eps,u,v)} $ instead of $\mathcal T^a[\R^d]{(x_\eps, y_\eps, u,v)} $ and split it into
\[ 	\mathcal{T}^a(x_\eps,y_\eps, u, v) = 
	\mathcal{T}^a[B_\rho](x_\eps,y_\eps,u,v)+\mathcal{T}^a[B_\rho^c](x_\eps,y_\eps,u,v).\]
Using the measure decomposition as in the proof of Theorem \ref{thm:Lip} with the total variation measure satisfying 
$|K(x_\eps,z)- K(y_\eps,z)| =K_+^a(z)+K_-^a(z)$,
and in view of \eqref{eq:mes-decom} and the above inequalities, the following estimates hold. 
\begin{eqnarray*}
	\mathcal{T}^a[B_\rho](x_\eps,y_\eps, u,v) & \leq & 
	\int_{B_\rho}(\phi_\eps(a_\eps+z) - \phi_\eps(a_\eps) - D\phi_\eps(a_\eps) \cdot z)K_+^a(z)\;dz+\\&&
		\int_{B_\rho}(\psi_\zeta(x_\eps+z) - \psi_\zeta(a_\eps) - D\psi_\zeta(a_\eps) \cdot z)K_+^a(z)\;dz+\\&&
	\int_{B_\rho}(\phi_\eps(a_\eps -z) - \phi_\eps(a_\eps) + D\phi_\eps(a_\eps)\cdot z)K_-^a(z)\, dz\\ &= &
	\frac{1}{\eps^2}  \int_{B_\rho}  |z|^2 \; |K^a(x_\eps,z)-K^a(y_\eps,z)|\, dz + o_\zeta(1),
\end{eqnarray*}	
which in view of assumption $(K3)$ and of the choice of $\varphi$, is further bounded above by
\[\mathcal{T}^a[B_\rho](x_\eps,y_\eps, u,v)\leq  
	C_K \frac{1}{\eps^2} |a_\eps|^\gamma \rho + o_\zeta(1) \leq 
	C_K \eps^{2\gamma-2}\rho + o_\zeta(1).\]
Similarly, we obtain 
\begin{eqnarray*}
	\mathcal{T}^a[B_\rho^c](x_\eps,y_\eps, u,v) & \leq & 
	\int_{B_\rho^c}(\phi_\eps(a_\eps+z) - \phi_\eps(a_\eps) )K_+^a(z)\;dz - 
	\int_{B\setminus B_\rho}D\phi_\eps(a_\eps) \cdot z \; K_+^a(z)\;dz\\&&
	\int_{B_\rho^c}(\phi_\eps(a_\eps -z) - \phi_\eps(a_\eps) )K_-^a(z)\;dz + 
	\int_{B\setminus B_\rho}D\phi_\eps(a_\eps) \cdot z\; K_-^a(z)\;dz\\&\leq&
	2 || \phi_\eps ||_\infty \int_{B_\rho^c}   |K^a(x_\eps,z)-K^a(y_\eps,z)|\; dz+ \\ &&
	2 |D\phi_\eps(a_\eps)| \int_{B\setminus B_\rho}  |z|\; |K^a(x_\eps,z)-K^a(y_\eps,z)|\; dz+ o_\zeta(1), 
\end{eqnarray*}	
which in view of assumption $(K3)$ and of the choice of $\varphi$, is further bounded above by
\begin{eqnarray*}
	 \mathcal{T}^a[B_\rho^c](x_\eps,y_\eps, u,v)  & \leq &
	2 C_K( || \varphi ||_\infty |a_\eps|^\gamma \rho^{-1}+ 
	 || \varphi' ||_\infty |p| |a_\eps|^\gamma |\ln(|\rho|)| ) + o_\zeta(1)\\ &\leq &
	C \Big(  \eps^{2\gamma}  \rho^{-1}  +  \eps^{2\gamma} |\ln(|\rho|)| \Big)+ o_\zeta(1).	 
\end{eqnarray*}	
Putting together all the previous estimates and taking $\rho = \rho_0 \eps^{2r}$ with $r<\gamma$, it follows that 
\[ \mathcal{T}^a(x_\eps,y_\eps, u, v) \leq C  \eps^{2\gamma}\Big( \eps^{2r-2} + |\ln(|\rho_0 \eps^{2r}|)| + \eps^{-2r} \Big) + o_\zeta(1) = o_\eps(1) + o_\zeta(1).\]
Going back to \eqref{eq:comp-visc}, 
\[ 0<\delta \frac{M}{2}\leq \delta u(x_\eps)-\delta v(y_\eps) - \nu \leq  o_\eps(1) + o_\zeta(1),\]
and letting $\eps$, $\zeta$ and $\nu$ go to zero we arrive to a contradiction.
\end{proof}

The proof previously shown applies literally to parabolic integro-differential equations and the following theorem holds.

\begin{prop}\label{prop:comp-ev}
Let  $(K^a)_{a\in\mathcal A}$,  $(b^a)_{a\in\mathcal A}$, $(f^a)_{a\in\mathcal A}$ satisfy the same assumptions as in Proposition \ref{prop:comp}.  If $u\in USC(\R^d\times[0,T])$ and $v\in LSC(\R^d\times[0,T])$  are respectively  a bounded viscosity subsolution and a bounded viscosity supersolution of equation \eqref{eq:HJ} with $\mathcal H$ given by \eqref{eq:H} such that  for all times $t\in[0,T]$, $u(\cdot,t)\in C^{0,1}(\R^d)$ or $v(\cdot,t) \in \mathcal C^{0,1}(\R^d)$  and $u(x,0)\leq v(x,0)$, then $u(\cdot,t)\leq v(\cdot,t)$ for all $t \in [0,T]$.
\end{prop}

\noindent {\bf Acknowledgements.} 
Adina Ciomaga was partially supported by the ANR project ANR-16- CE40-0015-01.
Erwin Topp was partially supported by  Conicyt PIA Grant No. 79150056 and Foncecyt Iniciaci\'on No. 11160817. 
Daria Ghilli was partially supported by the Starting Grant 2015 Cariparo "Nonlinear partial differential equations: asymptotic problems and mean-field games".

\bibliography{CGT-biblio}

\begin{thebibliography}{10}

\bibitem{Ab01}
O.~Alvarez and M.~Bardi.
\newblock Viscosity solutions methods for singular perturbations in
  deterministic and stochastic control.
\newblock {\em SIAM J. Control Optim.}, 40(4):1159--1188, 2001/02.

\bibitem{A09}
M.~Arisawa.
\newblock Homogenization of a class of integro-differential equations with
  {L}\'{e}vy operators.
\newblock {\em Comm. Partial Differential Equations}, 34(7-9):617--624, 2009.

\bibitem{A12}
M.~Arisawa.
\newblock Homogenizations of integro-differential equations with {L}\'{e}vy
  operators with asymmetric and degenerate densities.
\newblock {\em Proc. Roy. Soc. Edinburgh Sect. A}, 142(5):917--943, 2012.

\bibitem{Bcd97}
M.~Bardi and I.~Capuzzo-Dolcetta.
\newblock {\em Optimal control and viscosity solutions of
  {H}amilton-{J}acobi-{B}ellman equations}.
\newblock Systems \& Control: Foundations \& Applications. Birkh\"{a}user
  Boston, Inc., Boston, MA, 1997.
\newblock With appendices by Maurizio Falcone and Pierpaolo Soravia.

\bibitem{Bcs16}
M.~Bardi, A.~Cesaroni, and A.~Scotti.
\newblock Convergence in multiscale financial models with non-{G}aussian
  stochastic volatility.
\newblock {\em ESAIM Control Optim. Calc. Var.}, 22(2):500--518, 2016.

\bibitem{Bct19}
M.~Bardi, A.~Cesaroni, and E.~Topp.
\newblock Cauchy problem and periodic homogenization for nonlocal
  hamilton-jacobi equations with coercive gradient terms.
\newblock {\em To appear in Proc. R. Soc. Edinb. A.}, 2019.

\bibitem{Bt15}
M.~Bardi and G.~Terrone.
\newblock Periodic homogenization of deterministic control problems via limit
  occupational measures.
\newblock In {\em Dynamics, games and science}, volume~1 of {\em CIM Ser. Math.
  Sci.}, pages 105--116. Springer, Cham, 2015.

\bibitem{Bcci12}
G.~Barles, E.~Chasseigne, A.~Ciomaga, and C.~Imbert.
\newblock Lipschitz regularity of solutions for mixed integro-differential
  equations.
\newblock {\em J. Differential Equations}, 252(11):6012--6060, 2012.

\bibitem{Bcci14}
G.~Barles, E.~Chasseigne, A.~Ciomaga, and C.~Imbert.
\newblock Large time behavior of periodic viscosity solutions for uniformly
  parabolic integro-differential equations.
\newblock {\em Calc. Var. Partial Differential Equations}, 50(1-2):283--304,
  2014.

\bibitem{Bci11}
G.~Barles, E.~Chasseigne, and C.~Imbert.
\newblock H\"older continuity of solutions of second-order non-linear elliptic
  integro-differential equations.
\newblock {\em J. Eur. Math. Soc. (JEMS)}, 13(1):1--26, 2011.

\bibitem{Bi08}
G.~Barles and C.~Imbert.
\newblock Second-order elliptic integro-differential equations: viscosity
  solutions' theory revisited.
\newblock {\em Ann. Inst. H. Poincar\'e Anal. Non Lin\'eaire}, 25(3):567--585,
  2008.

\bibitem{Bklt15}
G.~Barles, S.~Koike, O.~Ley, and E.~Topp.
\newblock Regularity results and large time behavior for integro-differential
  equations with coercive {H}amiltonians.
\newblock {\em Calc. Var. Partial Differential Equations}, 54(1):539--572,
  2015.

\bibitem{Bp87}
G.~Barles and B.~Perthame.
\newblock Discontinuous solutions of deterministic optimal stopping time
  problems.
\newblock {\em RAIRO Mod\'{e}l. Math. Anal. Num\'{e}r.}, 21(4):557--579, 1987.

\bibitem{Bp88}
G.~Barles and B.~Perthame.
\newblock Exit time problems in optimal control and vanishing viscosity method.
\newblock {\em SIAM J. Control Optim.}, 26(5):1133--1148, 1988.

\bibitem{Bs18}
E.~Bayraktar and Q.~Song.
\newblock Solvability of the nonlinear {D}irichlet problem with
  integro-differential operators.
\newblock {\em SIAM J. Control Optim.}, 56(1):292--315, 2018.

\bibitem{BKR01}
F.~E. Benth, K.~H. Karlsen, and K.~Reikvam.
\newblock Optimal portfolio selection with consumption and nonlinear
  integro-differential equations with gradient constraint: a viscosity solution
  approach.
\newblock {\em Finance Stoch.}, 5(3):275--303, 2001.

\bibitem{Cs09}
L.~Caffarelli and L.~Silvestre.
\newblock Regularity theory for fully nonlinear integro-differential equations.
\newblock {\em Comm. Pure Appl. Math.}, 62(5):597--638, 2009.

\bibitem{CSW05}
L.~A. Caffarelli, P.~E. Souganidis, and L.~Wang.
\newblock Homogenization of fully nonlinear, uniformly elliptic and parabolic
  partial differential equations in stationary ergodic media.
\newblock {\em Comm. Pure Appl. Math.}, 58(3):319--361, 2005.

\bibitem{CS10}
L.A. Caffarelli and P.~E. Souganidis.
\newblock Rates of convergence for the homogenization of fully nonlinear
  uniformly elliptic pde in random media.
\newblock {\em Invent. Math.}, 180(2):301--360, 2010.

\bibitem{CLd16}
H.~Chang-Lara and G.~D\'{a}vila.
\newblock H\"{o}lder estimates for non-local parabolic equations with critical
  drift.
\newblock {\em J. Differential Equations}, 260(5):4237--4284, 2016.

\bibitem{C12}
A.~Ciomaga.
\newblock On the strong maximum principle for second order nonlinear parabolic
  integro-differential equations.
\newblock {\em Advances in Differential Equations}, 17(7-8):635--671, 2012.

\bibitem{Co65}
P.~Courr\`ege.
\newblock Sur la forme int\'{e}gro-diff\'{e}rentielle des op\'{e}rateurs de
  $c_k^\infty$ dans c satisfaisant au principe du maximum. s\'{e}minaire
  brelot-choquet-deny.
\newblock {\em Th\'{e}orie du Potentiel}, 10(1):1--38, 1965.

\bibitem{DNpv12}
E.~Di~Nezza, G.~Palatucci, and E.~Valdinoci.
\newblock Hitchhiker's guide to the fractional {S}obolev spaces.
\newblock {\em Bull. Sci. Math.}, 136(5):521--573, 2012.

\bibitem{E89}
L.~C. Evans.
\newblock The perturbed test function method for viscosity solutions of
  nonlinear {PDE}.
\newblock {\em Proc. Roy. Soc. Edinburgh Sect. A}, 111(3-4):359--375, 1989.

\bibitem{E92}
L.~C. Evans.
\newblock Periodic homogenisation of certain fully nonlinear partial
  differential equations.
\newblock {\em Proc. Roy. Soc. Edinburgh Sect. A}, 120(3-4):245--265, 1992.

\bibitem{Frs17}
J.~Fern\'{a}ndez~Bonder, A.~Ritorto, and A.~M. Salort.
\newblock {$H$}-convergence result for nonlocal elliptic-type problems via
  {T}artar's method.
\newblock {\em SIAM J. Math. Anal.}, 49(4):2387--2408, 2017.

\bibitem{I89}
H.~Ishii.
\newblock On uniqueness and existence of viscosity solutions of fully nonlinear
  second-order elliptic {PDE}s.
\newblock {\em Comm. Pure Appl. Math.}, 42(1):15--45, 1989.

\bibitem{Kpz19}
M.~Kassmann, A.~Piatnitski, and E.~Zhizhina.
\newblock Homogenization of {L}\'{e}vy-type operators with oscillating
  coefficients.
\newblock {\em SIAM J. Math. Anal.}, 51(5):3641--3665, 2019.

\bibitem{Lpv86}
P.L. Lions, G.~Papanicolaou, and S.R.S. Varadhan.
\newblock Homogenization of hamilton-jacobi equations.
\newblock {\em Manuscript}, 1986.

\bibitem{Pz17}
A.~Piatnitski and E.~Zhizhina.
\newblock Periodic homogenization of nonlocal operators with a convolution-type
  kernel.
\newblock {\em SIAM J. Math. Anal.}, 49(1):64--81, 2017.

\bibitem{S09}
R.~W. Schwab.
\newblock Stochastic homogenization of {H}amilton-{J}acobi equations in
  stationary ergodic spatio-temporal media.
\newblock {\em Indiana Univ. Math. J.}, 58(2):537--581, 2009.

\bibitem{S10}
R.~W. Schwab.
\newblock Periodic homogenization for nonlinear integro-differential equations.
\newblock {\em SIAM J. Math. Anal.}, 42(6):2652--2680, 2010.

\bibitem{S11}
L.~Silvestre.
\newblock On the differentiability of the solution to the {H}amilton-{J}acobi
  equation with critical fractional diffusion.
\newblock {\em Adv. Math.}, 226(2):2020--2039, 2011.

\bibitem{S12}
L.~Silvestre.
\newblock H\"{o}lder estimates for advection fractional-diffusion equations.
\newblock {\em Ann. Sc. Norm. Super. Pisa Cl. Sci. (5)}, 11(4):843--855, 2012.

\bibitem{S86a}
H.~M. Soner.
\newblock Optimal control with state-space constraint. {I}.
\newblock {\em SIAM J. Control Optim.}, 24(3):552--561, 1986.

\bibitem{S86b}
H.~M. Soner.
\newblock Optimal control with state-space constraint. {II}.
\newblock {\em SIAM J. Control Optim.}, 24(6):1110--1122, 1986.

\bibitem{T11}
G.~Terrone.
\newblock Limiting relaxed controls and averaging of singularly perturbed
  deterministic control systems.
\newblock {\em Dyn. Contin. Discrete Impuls. Syst. Ser. A Math. Anal.},
  18(5):653--672, 2011.

\end{thebibliography}
\bibliographystyle{plain} 

\end{document}